\documentclass[12pt,a4paper,reqno]{amsart}
\usepackage[left=2cm, right=2cm, bottom=3cm]{geometry} %Test zum Abändern der Seitenränder und Seitenhöhen

\usepackage[ngerman, english]{babel}

\usepackage{palatino}
\usepackage{amsmath}
\usepackage{amsthm}
\usepackage{amssymb}
\usepackage{url}
\usepackage{hyperref}
\usepackage{graphicx}
\usepackage{url}
\usepackage{listings}
\makeindex

\allowdisplaybreaks

\newtheorem{Satz}{Theorem}[section]
\newtheorem{Theorem}[Satz]{Theorem}
\newtheorem{Lemma}[Satz]{Lemma}		
\newtheorem{Korollar}[Satz]{Corollary}	
\newtheorem{Prop}[Satz]{Proposition}	
\numberwithin{equation}{section}
\theoremstyle{definition}

\newtheorem{Definition}[Satz]{Definition}

\newtheorem{Remark}[Satz]{Remark}

\newtheorem{Assumption}[Satz]{Assumption}
\newtheorem{Notation}[Satz]{Notation}

\newcommand{\BIGOP}[1]{\mathop{\mathchoice%
		{\raise-0.22em\hbox{\huge $#1$}}%
		{\raise-0.05em\hbox{\Large $#1$}}{\hbox{\large $#1$}}{#1}}}

\newcommand{\BIGboxplus}{\mathop{\mathchoice%
		{\raise-0.35em\hbox{\huge $\boxplus$}}%
		{\raise-0.15em\hbox{\Large $\boxplus$}}{\hbox{\large $\boxplus$}}{\boxplus}}}

\newcommand{\C}{\mathbb{C}} % komplexe
\newcommand{\R}{\mathbb{R}} % reelle
\newcommand{\Rd}{{\mathbb{R}^d}} % reelle
\newcommand{\N}{\mathbb{N}} % natuerliche

\newcommand{\E}{\mathbb{E}}
\newcommand{\Prob}{\mathbb{P}}
\newcommand{\F}{\mathcal{F}}
\newcommand{\Filtration}{\mathbb{F}}

\newcommand{\Etilde}{\tilde{\E}}
\newcommand{\HS}{\operatorname{HS}}

\newcommand{\Yosida}{R_\lambda}

\newcommand{\mass}{\mathcal{M}}

\newcommand{\Lzwei}{L^2(M)}

\newcommand{\Heins}{{H^1(M)}}

\newcommand{\Hminuseins}{{H^{-1}(M)}}

\newcommand{\Hminusdrei}{{H^{-3}(M)}}

\newcommand{\SchwachStetigHeins}{{C_w([0,T],\Heins)}}

\newcommand{\LInfty}{{L^\infty(M)}}

\newcommand{\LocTwo}[2]{{\psi_{#1,#2}(-h^2\Delta_g)}}
\newcommand{\Loc}{\varphi(h^2\Delta_g)}
\newcommand{\LocK}{\varphi(2^{-k}\Delta_g)}

\newcommand{\LocTilde}{\tilde{\varphi}(h^2\Delta_g)}

\newcommand{\groupMinusS}{e^{-\im  s\Delta_g}}
\newcommand{\group}{e^{\im  t\Delta_g}}
\newcommand{\groupTS}{e^{\im  (t-s)\Delta_g}}
\newcommand{\realTest}{C_c^\infty(\R)}
\newcommand{\realTestNoNull}{C_c^\infty(\R \setminus \{0\})}
\newcommand{\supp}{\operatorname{supp}}
\newcommand{\JendpointSpace}{{L^2(J,L^6)}}
\newcommand{\IjStrichendpointSpace}{{L^2(I_j',L^6)}}
\newcommand{\JjStrichendpointSpace}{{L^2(J_j',L^6)}}
\newcommand{\JendpointSpaceGamma}{{\gamma(Y,L^2(J,L^6))}}

\newcommand{\df }{\mathrm{d}}
\newcommand{\im }{\mathrm{i}}
\newcommand{\sumM }{\sum_{m=1}^{\infty}}
\newcommand{\Real}{\operatorname{Re}}
\newcommand{\skpLzwei}[2]{\big(#1,#2\big)_{L^2}}

\newcommand{\skpLzweiMf}[2]{\big(#1,#2\big)_{L^2}}

\newcommand{\norm}[1]{\Vert #1 \Vert}
\newcommand{\bigNorm}[1]{\left\Vert #1 \right\Vert}

\newcommand{\Addresses}{{% additional braces for segregating \footnotesize
		\bigskip
		\footnotesize
		
		Z.~Brze{\'{z}}niak, Department of Mathematics, University of York,
			Heslington, York, YO105DD, UK\par\nopagebreak
		\textit{E-mail address}: zdzislaw.brzezniak@york.ac.uk
		\medskip
		
		F.~Hornung, Institute for Analysis, Karlsruhe Institute for Technology (KIT), 76128 Karlsruhe, Germany\par\nopagebreak
		\textit{E-mail address}: fabian.hornung@kit.edu
		
		\medskip
		
		L.~Weis, Institute for Analysis, Karlsruhe Institute for Technology (KIT), 76128 Karlsruhe, Germany\par\nopagebreak
		\textit{E-mail address}: lutz.weis@kit.edu
		
	}}

\title[Uniqueness of  solutions for the stochastic 3d NLS]{Uniqueness of martingale solutions for the stochastic nonlinear Schr\"odinger equation on 3d compact manifolds}
\author{Zdzislaw BRZE{\'Z}NIAK, Fabian HORNUNG AND Lutz WEIS
	}

\date{\today}

\begin{document}

%\begin{abstract}
%	We consider a stochastic nonlinear Schr\"{o}dinger equation with multiplicative noise in an abstract framework that
%	covers subcritical focusing and defocusing SNLS in $H^1$ on compact manifolds and bounded domains and construct a martingale solution using a modified Faedo-Galerkin-method based on general Paley-Littlewood-decompositions. In the case of 2d compact manifolds, we use Strichartz estimates to show uniqueness.
%\end{abstract}
\begin{abstract}
	We prove pathwise uniqueness for solutions of the nonlinear Schr\"{o}dinger equation with conservative multiplicative noise on compact 3D manifolds. In particular, we generalize the result by Burq, G\'erard and Tzvetkov, \cite{Burq}, to the stochastic setting. The proof is based on deterministic and stochastic Strichartz estimates and the Littlewood-Paley decomposition.
\end{abstract}

\maketitle

%\begin{keyword}
%	Stochastic Nonlinear Schr\"{o}dinger Equation, Galerkin Method
%\end{keyword}

%\keywords{Nonlinear Schr\"{o}dinger Equation, Multiplicative Noise, Galerkin Method}

%\section{Dokumentenklassen} \label{documentclasses}
%
%\begin{itemize}
%\item article
%\item book
%\item report
%\item letter
%\end{itemize}
%
%
%\begin{enumerate}
%\item article
%\item book
%\item report
%\item letter
%\end{enumerate}
%
%\begin{description}
%\item[article\label{article}]{Article ist \ldots}
%\item[book\label{book}]{Die book Klasse ist \ldots}
%\item[report\label{report}]{Die Klasse report erm\"oglicht es  \ldots}
%\item[letter\label{letter}]{Wenn man einen Breif schreiben sollte man eine
%	andere Klasse nutzen, da diese f\"ur ein anderes als das deutsche
%	Briefformat ausgelegt ist.}
%\end{description}
%
%
%\section{Fazit}\label{conclusions}
%Nach langer Suche hat sich herausgestellt, dass es kein l\"angeres
%\LaTeX{} Beispiel, als das von \cite{doe} geschriebene gibt.
%
%\begin{thebibliography}{9}
%\bibitem[Doe]{doe} \emph{Erstes und letztes \LaTeX{} Beispiel.},
%John Doe 50 v.Chr.
%\end{thebibliography}

%\input{Introduction}
%\input{Setting}
%\input{Examples}	
%\input{Compactness}
%\input{Galerkin}
%\input{ExistenceNoVitali}
%\input{Uniqueness}	
%\input{Appendix}


\section{Introduction and main result}

This article is concerned with the nonlinear Schr\"odinger equation with multiplicative noise
\begin{equation}
\label{ProblemStratonovich}
%\begin{gathered}
%(\operatorname{SNLS})\end{gathered}
\left\{
\begin{aligned}
\df u(t)&= \left(\im \Delta_g u(t)-\im \lambda \vert u(t)\vert^{\alpha-1} u(t)\right) dt-\im \sumM e_m u(t) \circ \df \beta_m(t),\hspace{0,3 cm} t\in (0,T),\\
u(0)&=u_0\in \Heins,
\end{aligned}\right.
\end{equation}
on a compact riemannian manifold $M,$ where $\Delta_g$ is the Laplace-Beltrami-operator, $\alpha>1,$ $\lambda\in \{-1,1\},$ $\left(e_m\right)_{m\in\N}$ are real valued functions and  $\left(\beta_m\right)_{m\in\N}$ are independent Brownian motions. if $\lambda=1,$ the NLS is called defocusing and $\lambda=-1,$ it is called focusing.

In the previous article \cite{ExistencePaper}, we constructed a martingale solution of $\eqref{ProblemStratonovich}$ in arbitrary dimension for $\lambda=1$ and $\alpha\in (1,1+\frac{4}{(d-2)_+})$ or $\lambda=-1$ and $\alpha\in (1,1+\frac{4}{d}).$ Moreover, we proved pathwise uniqueness of solutions in the $2D$-case. The aim of the present article is to show pathwise uniqueness in the significantly harder three dimensional case and to generalize the result by Burq, G\'erard and Tzvetkov from \cite{Burq}, Theorem 3, for the cubic NLS  to the stochastic setting.

\begin{Theorem}\label{Uniqueness3d}
	Let $M$ be a compact $3D$ riemannian manifold.
	Let $\lambda\in \{-1,1\},$ $\alpha\in (1,3]$ and $e_m\in \LInfty$ real valued with $\nabla e_m\in L^3(M)$  for $m\in\N$ and
	\begin{align}\label{Regularityem}
	\sumM \left(\norm{\nabla e_m}_{L^3}+\norm{e_m}_{L^\infty}\right)^2<\infty.
	\end{align}
	Then, solutions of $\eqref{ProblemStratonovich}$ are pathwise unique.
\end{Theorem}
Note that in contrast to existence, the uniqueness result is the same for the focusing and defocusing NLS.
As an immediate consequence of the Yamada-Watanabe-Theory developed in \cite{KunzeYamada}, Theorem 5.3 and Corollary 5.4, we obtain the existence of a unique strong solution of \eqref{ProblemStratonovich}.

\begin{Korollar}\label{Korollar3d}
	Let $M$ be a compact $3D$ riemannian manifold.
	Let $\lambda=1$ and $\alpha\in (1,3]$ or $\lambda=-1$ and $\alpha \in (1,\frac{7}{3}).$ If $\left(e_m\right)_{m\in\N}$ satisfies the conditions from Theorem \ref{Uniqueness3d}, there is a unique strong solution of $\eqref{ProblemStratonovich}$ and martingale solutions are unique in law.
\end{Korollar}

%\begin{proof}
%	The existence of a martingale solution was proved in \textbf{TODO: Reference on existence paper}. In view of Theorem $\ref{Uniqueness3d},$ the assertion follows from  \cite{OndrejatUniqueness}, Theorem 2 and 12.1.
%\end{proof}

  The question of existence and uniqueness of global solutions of the stochastic nonlinear Schr\"odinger equation was previously addressed by de Bouard and Debussche in \cite{BouardLzwei} and \cite{BouardHeins}, Barbu, R\"ockner and Zhang in \cite{BarbuL2}, \cite{BarbuH1}, \cite{zhang2017}
   and Hornung in \cite{FHornung}. In these articles, the authors considered the fullspace $\Rd$ and employed a fixed point argument based on Strichartz estimates to prove existence and uniqueness in one step. As in the deterministic setting, their ranges of exponents $\alpha$ depend on the space dimension and the considered regularity. Brze{\'z}niak and Millet followed a similar approach for the stochastic NLS on a compact 2D manifold $M.$ In higher dimensions, their argument only yields local solutions  since the estimates for the nonlinearity rely on the Sobolev embeddings $H^{s,p}\hookrightarrow L^\infty$ that are too restrictive to work in the energy space $\Heins.$
%  To show uniqueness, we use the Strichartz estimates
%  \begin{align}\label{localizedStrichartz}
%  \norm{t\mapsto e^{\im  t\Delta_g}\Loc x}_{L^q(0,T;L^p)}\lesssim \norm{x}_{L^2},\qquad x\in \Lzwei.
%  \end{align}
%  and
%  \begin{align}\label{StrichartzLoss}
%  \norm{t\mapsto e^{\im  t\Delta_g}x}_{L^q(0,T;L^p(M))}\lesssim \norm{x}_{H^{\frac{1}{q}}(M)},\qquad x\in H^{\frac{1}{q}}(M).
%  \end{align}
%  from \cite{Burq} for $p,q\in [2,\infty]$ with
%  $
%  \frac{2}{q}+\frac{d}{p}=\frac{d}{2}$ and $(q,p,d)\neq(2,\infty,2).
%  $
  %(for the details, see Lemma \ref{homogenousStrichartzLemma}).
% Note that Keller and Lisei, see \cite{kellerLisei}, considered the stochastic NLS in 1D on the interval $(0,1)$ with Neumann boundary conditions.  They proved existence with a Galerkin method and uniqueness via the Sobolev embedding $H^1(0,1)\hookrightarrow L^\infty(0,1).$ Hence, their argument cannot be transfered to higher dimensions. \\
   Another result about the stochastic NLS is due to
   Keller and Lisei, see \cite{lisei2016stochastic}, who considered the equation on the space-interval $(0,1)$ with Neumann boundary conditions.  They proved existence with a Galerkin method and uniqueness via the Sobolev embedding $H^1(0,1)\hookrightarrow L^\infty(0,1).$ Hence, their argument cannot be transfered to higher dimensions.
   After this work was finished, we learned about a recent paper \cite{CheungMosincat} by Cheung and Mosincat. Using the additional structure in the special case of the $d$-dimensional torus $M=\mathbb{T}^d$ and algebraic nonlinearities, i.e. $\alpha=2k+1$ for some $k\in\N,$ the authors employed a fixed point argument based on multilinear Strichartz estimates and an estimate of the stochastic convolution in Bourgain spaces $X^{s,b}$ combined with the truncation method from  \cite{BouardLzwei}, \cite{BouardHeins} and \cite{FHornung}. As a result, they solved the NLS with multiplicative noise in $L^2(\Omega,C([0,\tau],H^s(\mathbb{T}^d))\cap X^{s,b}([0,\tau]))$ for all $s>s_{crit}:=\frac{d}{2}-\frac{2}{\alpha-1}$ and some $b<\frac{1}{2}$ as well as some stopping time $\tau>0.$ As a byproduct, their argument also implies pathwise uniqueness of martingale solutions in $L^2(\Omega,C([0,T],H^s(\mathbb{T}^3))\cap X^{s,b}([0,T]))$ for $\alpha=3$ and $s>\frac{1}{2},$ which reflects an improvement compared to the general case considered in Theorem \ref{Uniqueness3d}.  \\

    Our approach separates existence and uniqueness.
    The construction of a martingale solution  in \cite{ExistencePaper} did not use Strichartz estimates. It was only based on the Hamiltonian structure of the NLS and the compactness of the embedding $\Heins\hookrightarrow L^p(M).$ Since these ingredients are independent of the underlying geometry, the proof worked in a more general framework. In particular, we considered arbitrary dimensions $d\in \N$ and powers $\alpha\in (1,1+\frac{4}{(d-2)_+})$ and could also deal with Dirichlet and Neumann Laplacians on bounded domains as well as their fractional powers. The flexibility of this approach is underlined by the fact that it could be also used to construct a martingale solution of the NLS with pure jump noise, see \cite{NLSLevy}.
   In the following, we would like to explain the difficulties of the uniqueness result in the three dimensional case and sketch the proof, which is inspired by the ideas of Burq, G\'erard and Tzvetkov in \cite{Burq}.
   We take two solutions with $u_1,u_2\in L^\infty(0,T;\Heins)$ almost surely.
   Our starting point is the representation
   \begin{align}\label{differenceFormulaIntroduction}
   \norm{u_1(t)-u_2(t)}_{L^2}^2=&2 \int_0^t \Real \skpLzwei{u_1(s)-u_2(s)}{-\im \lambda\vert u_1(s)\vert^{\alpha-1} u_1(s)+\im \vert u_2(s)\vert^{\alpha-1} u_2(s)} \df s
   \end{align}
   almost surely for all $t\in [0,T].$ At this point, it is crucial to consider Stratonovich noise with real valued coefficients, since this leads to cancellations of the stochastic integral and the correction term in It\^o's formula. We remark that the formula \eqref{differenceFormulaIntroduction} is closely related to the mass conservation of solutions to \eqref{ProblemStratonovich} which leads to the notion of conservative noise.
   %  In 2D, one can use a Gronwall argument based on $H^{s-\frac{1}{q},p}(M)\hookrightarrow \LInfty$ for  $s\in (1-\frac{1}{q},1)$ and $u_1,u_2\in L^q(0,T; H^{s-\frac{1}{q},p})$  almost surely which follows from the Strichartz estimate \eqref{StrichartzLoss}.
   To use \eqref{differenceFormulaIntroduction} for a uniqueness proof, we employ
   %  we use the Strichartz estimates
   %  %(for the details, see Lemma \ref{homogenousStrichartzLemma}).
   %  one needs the additional regularity of solutions which one can get by
   the local Strichartz estimate
  \begin{align}\label{localizedStrichartz}
  \norm{t\mapsto e^{\im  t\Delta_g}\Loc x}_{L^q(0,T;L^p)}\lesssim \norm{x}_{L^2},\qquad x\in \Lzwei,
  \end{align}
   for small times $ T\lesssim h$ and the global Strichartz estimate
  \begin{align}\label{StrichartzLoss}
  \norm{t\mapsto e^{\im  t\Delta_g}x}_{L^q(0,T;L^p(M))}\lesssim \norm{x}_{H^{\frac{1}{q}}(M)},\qquad x\in H^{\frac{1}{q}}(M).
  \end{align}
   from \cite{Burq} for $p,q\in [2,\infty]$ with
   $
   \frac{2}{q}+\frac{d}{p}=\frac{d}{2}$ and $(q,p,d)\neq(2,\infty,2).$ Here, $h\in (0,1]$ and $\varphi \in C_c^\infty(\R)$ can be chosen arbitrarily.

   In two dimensions, \eqref{StrichartzLoss} improves the regularity to $u_1,u_2\in L^q(0,T; H^{s-\frac{1}{q},p})$ almost surely for  $s\in (1-\frac{1}{q},1).$
   Hence, one can use a Gronwall argument based on the Sobolev embedding $H^{s-\frac{1}{q},p}(M)\hookrightarrow \LInfty$  to prove pathwise uniqueness. For the details, we refer to \cite{ExistencePaper}.
   In 3D, the challenge is to gain $\frac{1}{2}+\varepsilon$ derivatives with respect to the embedding $H^{\frac{3}{2}+\varepsilon}(M)\hookrightarrow \LInfty$ in order to control the nonlinearity in \eqref{differenceFormulaIntroduction} by the $H^1$-estimates of the solutions. Unfortunately, this is not possible, but it turns out that one can replace $L^\infty$-estimates by
   \begin{align}\label{3DEstimateSolutions}
   \norm{u_j}_{L^2(J,L^p)}\lesssim 1+\left(\vert J\vert p\right)^\frac{1}{2}\qquad \text{a.s.}
   \end{align}
   for all $p\in [6,\infty)$ and intervals $J\subset [0,T].$ Then, we use \eqref{3DEstimateSolutions} and the control of the $L^p$-norms for $2\le p\le 6$ by $\Heins\hookrightarrow L^6(M)$ to get an inequality
   \begin{align}\label{YudovitchStrategy}
   \norm{u_1(t)-u_2(t)}_{L^2}^2\le C(p,u_1,u_2,\vert J\vert)
   \end{align}
   with $C(p,u_1,u_2,\vert J\vert)\to 0$ a.s. as $p\to \infty$ for sufficiently small time intervals $J\subset [0,T].$

      In order to get \eqref{3DEstimateSolutions}, we use partitions of unity to estimate the solutions locally in time and frequency by the Strichartz estimate \eqref{localizedStrichartz}.
      To control the stochastic term, we adapt Brze{\'z}niak's and Millet's approach from \cite{BrzezniakStrichartz} to derive a spectrally localized stochastic Strichartz estimate.
      Afterwards, we reassemble the local estimates by Littlewood-Paley-Theory. We point out that  the proof is restricted to dimension $d=3$ and $\alpha\in (1,3].$ In fact, we need the endpoint Strichartz estimate by Keel and Tao, \cite{keelTao}, to prove pathwise uniqueness for $\alpha=3.$
   We would like to point out that recently, Bernicot and Savoyeau, see \cite{Bernicot}, could prove estimates of the type of \eqref{localizedStrichartz} and \eqref{StrichartzLoss} also in the case of possibly non-compact manifolds with bounded geometry. Unfortunately, their estimate \eqref{localizedStrichartz} only holds for $T\le h^{1+\varepsilon}$ and \eqref{StrichartzLoss} holds with loss $\frac{1+\varepsilon}{p}$ for an arbitrary $\varepsilon>0.$ Moreover, the constants depend on $\varepsilon,$ which leads to an additional growth of the constant in \eqref{3DEstimateSolutions} as $p\to \infty.$ Hence, the results from \cite{Bernicot} cannot be applied scheme of proof.

  The strategy to use estimates of the type \eqref{YudovitchStrategy} to prove uniqueness was developed by Yudovitch, \cite{yudovich1963non}, for the Euler equation. In the context of the NLS, it was  used by Vladimirov in \cite{Vladimirov}, Ogawa and Ozawa in \cite{Ogawa} and \cite{OgawaOzawa}. They looked at $2D$ domains and used Trudinger type inequalities as an analogon to \eqref{3DEstimateSolutions} to control the growth of $L^p$-norms for $p\to \infty.$
   Burq, G\'erard and Tzvetkov could use the Yudovitch-strategy for three dimensional manifolds without boundary due to the regularizing effect of Strichartz estimates. In \cite{blairStrichartz}, Blair, Smith and Sogge proved uniqueness of weak solutions of the deterministic NLS on compact $3D$ manifolds with boundary as an application of their Strichartz estimates on this type of geometry.

  The paper is organized as follows. In section 2, we fix the notations, formulate our assumptions and collect auxiliary results. Section 3 is devoted to proof of the estimate \eqref{3DEstimateSolutions} and the pathwise uniqueness. 
\section{Definitions and auxiliary results}

This section is devoted to the notations, definitions and auxiliary results that will be used in the next section to show pathwise uniqueness.

If $a,b\ge 0$ satisfy the inequality $a\le C b$ with a constant $C>0$, we write $a \lesssim b.$ Given $a\lesssim b$ and $b\lesssim a,$ we write $a\eqsim b.$
For two   Banach spaces $E,F$, we denote by $\mathcal{L}(E,F)$ the space of linear bounded operators $B: E\to F$ and abbreviate $\mathcal{L}(E):=\mathcal{L}(E,E).$
We use the notation $\HS(H_1,H_2)$ for the space of Hilbert-Schmidt-operators between Hilbert spaces $H_1$ and $H_2.$
Furthermore, we write $E\hookrightarrow F$ if $E$ is continuously embedded in $F;$ i.e. $E\subset F$ with natural embedding  $j\in \mathcal{L}(E,F).$

Let $M$ be a three dimensional compact riemannian $C^\infty$ manifold without boundary and  $L^p(M)$ for $p\in [1,\infty]$ the space equivalence classes of $\C$-valued  $p$-integrable functions. The distance induced by $g$ is denoted by $\rho$ and canonical measure on $M$ is called $\mu.$
By $L^p(M)$ for $p\in [1,\infty],$ we denote the space of equivalence classes of $\C$-valued  $p$-integrable functions with respect to $\mu.$  The Laplace-Beltrami operator on $M,$ i.e. the generator of the heat semigroup on $M$, is named $\Delta_g.$  Moreover, we use the fractional Sobolev spaces
%$H^{s,p}(M):=\D\left((I-\Delta_g)^\frac{s}{2}\right)$
\begin{align*}
H^{s,p}(M):=\left\{u\in L^p(M): \exists v\in L^p(M): u=(I-\Delta_g)^{-\frac{s}{2}}v \right\}
\end{align*}
for $p\in [1,\infty)$ and $s\ge 0$
with the norm
$\norm{u}_{H^{s,p}}:=\norm{v}_{L^p}.$
For $s< 0,$ the space  $H^{s,p}(M)$ is defined as the completion of $L^p(M)$ with respect to
\begin{align*}
\norm{u}_{H^{s,p}}:=\norm{(I-\Delta_g)^{\frac{s}{2}}u}_{L^p},\qquad u\in L^p(M).
\end{align*}	
For all $s\in \R,$ we shortly denote $H^s(M):=H^{s,2}(M).$ For properties of the Laplace-Beltrami operator, characterizations of the fractional Sobolev spaces and embedding theorems, we refer to \cite{Triebel} and \cite{Strichartz}. For $s=1,$ one can show that the definition from above coincides with the classical Sobolev space and $\left(\norm{u}_{L^2}^2+\norm{\nabla u}_{L^2}^2\right)^\frac{1}{2}$ defines an equivalent norm on $\Heins.$ We refer to \cite{lablee2015spectral} for an explanation of the gradient as an element of the tangential bundle of $M.$
%We often abbreviate $L^p:=L^p(M)$ and $H^{s,p}:=H^{s,p}(M).$\\

Next, we summarize the assumptions on the coefficient of the noise in $\eqref{ProblemStratonovich}.$
\begin{Assumption}\label{stochasticAssumptions}
	Let $Y$ be a separable Hilbert space  and $B: \Heins \to \HS(Y,\Heins)$ a linear operator. For an ONB $\left(f_m\right)_{m\in\N}$ of $Y$ and $m\in \N,$ we set $B_m :=B(\cdot)f_m.$ Additionally, we assume that $B_m,$ $m\in\N,$ are bounded operators on $\Heins$ with
	\begin{align}\label{noiseBoundsH}
	\sumM \norm{B_m}_{\mathcal{L}(H^1)}^2<\infty
	\end{align}
	and that $B_m$ is symmetric as operator in $\Lzwei,$ i.e.
	\begin{align}\label{BmSymmetric}
		\skpLzwei{B_m u}{v}=\skpLzwei{u}{B_m v},\qquad u,v\in\Heins.
	\end{align}
\end{Assumption}
%\begin{Assumption}\label{stochasticAssumptions}
%	Let $Y$ be a separable Hilbert space  and $B\in \mathcal{L}(\Heins, \HS(Y,\Heins)).$ For an ONB $\left(f_m\right)_{m\in\N}$ of $Y$ and $m\in \N,$ we set $B_m :=B(\cdot)f_m$ and
%	\begin{align*}
%	\mu := -\frac{1}{2} \sumM B_m^2.
%	\end{align*}	
%	Additionally, we assume $\mu\in\mathcal{L}(\Heins)$ and
%	\begin{align*}
%		\skpLzwei{B_m u}{v}=\skpLzwei{u}{B_m v},\qquad  u,v\in \Heins,m\in \N.
%	\end{align*}
%\end{Assumption}
In particular, we have $B\in \mathcal{L}\left(\Heins,\HS(Y,\Heins)\right)$ and $\mu\in \mathcal{L}(\Heins)$ if we abbreviate
\begin{align*}
	\mu(u) := -\frac{1}{2} \sumM B_m^2 u,\qquad u\in\Heins.
\end{align*}
We look at the following slight generalization of \eqref{ProblemStratonovich} in the It\^o form 	
\begin{equation}\label{Problem}
%\begin{gathered}
%(\operatorname{SNLS})\end{gathered}
\left\{
\begin{aligned}
\df u(t)&= \left(\im \Delta_g u(t)-\im \lambda \vert u(t)\vert^{\alpha-1} u(t)+ \mu \left(u(t)\right) \right) \df t-\im B u(t) \df W(t),\hspace{0,3 cm} t\in (0,T),\\
u(0)&=u_0.
\end{aligned}\right.
\end{equation}
In the introduction, we used that the process
\begin{align*}
W=\sumM f_m \beta_m
\end{align*}
with a sequence $\left(\beta_m\right)_{m\in\N}$ of independent Brownian motions is a cylindrical Wiener process in $Y,$ see \cite{daPrato}, Proposition 4.7, and the identity
\begin{align}
-\im B u(t) \circ\df W(t)=-\im B u(t) \df W(t)+\mu\left(u(t)\right) \df t,
\end{align}
which relates It\^o and Stratonovich noise. For the sake of simplicity, we restricted ourselves to the special case of multiplication operators
\begin{align*}
	B_m u=e_m u, \qquad u\in\Heins.
\end{align*}
with real valued functions $e_m$ satisfying
	\begin{align}\label{emAssumption}
	\sumM \left(\norm{\nabla e_m}_{L^3}+\norm{e_m}_{L^\infty}\right)^2<\infty.
	\end{align}
 We want to justify that they fit in  Assumption \ref{stochasticAssumptions}.
The Sobolev embedding $\Heins \hookrightarrow L^6(M)$  and the H\"older inequality yield
\begin{align*}
\norm{\nabla \left(e_m u\right)}_{L^2} \le& \norm{ u \nabla e_m }_{L^2}+ \norm{ e_m \nabla u }_{L^2}
\le \norm{\nabla e_m}_{L^3} \norm{u}_{L^6}+\norm{e_m}_{L^\infty} \norm{\nabla u}_{L^2}\\
\lesssim& \left(\norm{\nabla e_m}_{L^3}+\norm{e_m}_{L^\infty}\right) \norm{u}_{H^1},\qquad u\in\Heins.
\end{align*}
Thus,
\begin{align*}
	\norm{B_m u}_{H^1}\eqsim \norm{e_m u}_{L^2}+\norm{\nabla \left(e_m u\right)}_{L^2}\lesssim \left(\norm{\nabla e_m}_{L^3}+\norm{e_m}_{L^\infty}\right) \norm{u}_{H^1},\qquad u\in\Heins.
\end{align*}
Note that the existence-Theorem from \cite{ExistencePaper}  additionally needs the assumptions $B_m\in \mathcal{L}(\Lzwei)\cap \mathcal{L}(L^{\alpha+1}(M))$ with
	\begin{align}
	\sumM \norm{B_m}_{\mathcal{L}(L^2)}^2<\infty,\qquad \sumM \norm{B_m}_{\mathcal{L}(L^{\alpha+1})}^2<\infty.
	\end{align}
But in our example of multiplication operators, this assumption is implied by \eqref{emAssumption}.	
 In the first Definition, we explain two solution concepts for problem \eqref{ProblemStratonovich}.

\begin{Definition}\label{MartingaleSolutionDef}
	Let $T>0$ and $u_0\in \Heins.$
	\begin{enumerate}
		\item[a)] A \emph{martingale solution} of the equation $\eqref{ProblemStratonovich}$ is a system $\left(\Omega,\F,\Prob,W,\Filtration,u\right)$ with
		\begin{itemize}
			\item a probability space $\left(\Omega,\F,\Prob\right)$
			\item a $Y$-valued cylindrical Wiener $W$ process on $\Omega;$
			\item  a filtration $\Filtration=\left(\F_t\right)_{t\in [0,T]}$ with the usual conditions;
			\item a continuous, $\Filtration$-adapted process with values in $\Hminuseins$ such that almost all paths are in $\SchwachStetigHeins$ and $u\in L^2(\Omega\times [0,T],\Heins);$
		\end{itemize}
		such that the equation
		\begin{align}\label{ItoFormSolution}
		u(t)=  u_0+ \int_0^t \left[\im \Delta_g u(s)-\im \lambda \vert u(s)\vert^{\alpha-1} u(s)+\mu(u(s))\right] \df s- \im \int_0^t B u(s) \df W(s)
		\end{align}
		holds almost surely in $\Hminuseins$ for all $t\in [0,T].$
		\item[b)]
		%	  $\left(\Omega,\F,\Prob}\right)$  be a probability space,  $W$
		%	a $Y$-valued cylindrical Wiener process on $\Omega$ and  $\Filtration=\left(\F_t\right)_{t\in [0,T]}$
		%	 a filtration  with the usual conditions.	
		Given a probability space $\left(\Omega,\F,\Prob\right),$
		a $Y$-valued cylindrical Wiener $W$ process on $\Omega,$
		and a filtration $\Filtration=\left(\F_t\right)_{t\in [0,T]}$ with the usual conditions,
		a \emph{strong solution} of the equation $\eqref{ProblemStratonovich}$ is  a continuous, $\Filtration$-adapted process with values in $\Hminuseins$ such that %$u\in L^2(\Omega\times [0,T],\Hminuseins),$
		almost all paths are in $\SchwachStetigHeins,$ $u\in L^2(\Omega\times [0,T],\Heins)$
		and $\eqref{ItoFormSolution}$ holds
		almost surely in $\Hminuseins$ for all $t\in [0,T].$
	\end{enumerate}
	
\end{Definition}

\begin{Remark}
	For $\alpha\in (1,3],$ the solution is almost surely continuous in $\Lzwei.$
	Indeed, this follows from the mild form
	\begin{align}\label{MildFormSolution}
	u(t)=  \group u_0+ \int_0^t \groupTS \left[-\im \lambda\vert u(s)\vert^{\alpha-1} u(s)+\mu(u(s))\right] \df s- \im \int_0^t \groupTS B u(s) \df W(s)
	\end{align}
	almost surely for all $t\in [0,T]$ (see for example the proof of Proposition $\ref{controlHighNorms}$ in a similar situation), since the nonlinearity with $\alpha\in (1,3]$ maps $\Heins$ to $\Lzwei$ by the Sobolev embedding $\Heins\hookrightarrow L^{2\alpha}(M).$
	%	 the mild equation $\eqref{MildFormSolution}$ yields that almost all paths of solutions of $\eqref{ProblemStratonovich}$ are continuous in $\Lzwei.$
\end{Remark}

In the following definition, we fix different notions of uniqueness.
As we have seen in the previous remark, it makes sense to define uniqueness by comparing solutions in $C([0,T],\Lzwei).$

\begin{Definition}
	\begin{enumerate}
		\item[a)] 		The solutions of problem $\eqref{ProblemStratonovich}$ are called \emph{pathwise unique in \\$L^2(\Omega;L^\infty(0,T;\Heins))$}, if  given two martingale solutions $\left(\Omega,\F,\Prob,W,\Filtration,u_j\right)$ with $u_j\in L^2(\Omega;L^\infty(0,T;\Heins))$
		%		with
		%		\begin{align*}
		%		\exists  r>\max\{2\alpha, (\alpha-1)\alpha\}: u_j\in L^{r}(\tilde{\Omega},L^\infty(0,T;H^1(M))),
		%		\end{align*}
		for $j=1,2,$  we have $u_1(t)=u_2(t)$ almost surely in ${L^2(M)}$ for all $t\in [0,T].$
		\item[b)] The solutions of $\eqref{ProblemStratonovich}$ are called \emph{unique in law in $L^2(\Omega;L^\infty(0,T;\Heins))$}, if given two martingale solutions $\left(\Omega_j,\F_j,\Prob_j,W_j,\Filtration_j,u_j\right)$  with $u_j(0)=u_0$ and \\$u_j\in L^2(\Omega;L^\infty(0,T;\Heins))$ for $j=1,2,$ we have $\Prob_1^{u_1}=\Prob_2^{u_2}$  on $C([0,T],L^2(M)).$
	\end{enumerate}
	
\end{Definition}

We continue with some auxiliary results which are either well-known or due to Burq, G\'erard and Tzvetkov, \cite{Burq}.
The first Lemma gives us an estimate for the nonlinear term in $\eqref{ProblemStratonovich}.$

\begin{Lemma}\label{LemmaEstimateNonlinearity}
	Let $q\in [2,6]$ and $r\in (1,\infty)$ with $\frac{1}{r'}=\frac{1}{2}+\frac{\alpha-1}{q}.$ Then, we have
	\begin{align*}
		\norm{\vert u\vert^{\alpha-1}u}_{H^{1,r'}}\lesssim \norm{u}_{H^1}^\alpha,\qquad u\in \Heins.
	\end{align*}
\end{Lemma}
\begin{proof}
	See \cite{Bolleyer}, Lemma III.1.4.
\end{proof}

 The following Lemma deals with a Littlewood-Paley type decomposition of $L^p(M)$ for $p\in [2,\infty).$

\begin{Lemma}\label{LittlewoodPaley}
		Let $\psi \in \realTest,$ $\varphi\in \realTestNoNull$ with
		\begin{align*}
		1=\psi(\lambda)+\sum_{k=1}^\infty \varphi(2^{-k}\lambda),\qquad\lambda\in\R.
		\end{align*}
		Then, we have
		\begin{align}\label{LittlewoodPaleyEquHilbert}
		\norm{f}_{L^2}\eqsim \left( \norm{\psi(\Delta_g)f}_{L^2}^2+\sum_{k=1}^\infty \norm{\varphi(2^{-k}\Delta_g) f }_{L^2}^2\right)^\frac{1}{2},\qquad f\in \Lzwei,
		\end{align}
		and
		\begin{align}\label{LittlewoodPaleyEstLp}
		\norm{f}_{L^p}\lesssim_p \norm{\psi(\Delta_g)f}_{L^p}+\left(\sum_{k=1}^\infty \norm{\varphi(2^{-k}\Delta_g) f }_{L^p}^2\right)^\frac{1}{2},\qquad f\in L^p(M),
		\end{align}
		for $p\in [2,\infty).$
\end{Lemma}

\begin{proof}
	Let $p\in (1,\infty).$ By \cite{Bouclet}, page 2, or \cite{KrieglerWeis} Theorem 4.1 and estimate (2.9) in a more general setting, we have
	\begin{align*}
		\norm{f}_{L^p}\eqsim \bigNorm{\left(\vert\psi(\Delta_g)f\vert^{2}+\sum_{k=1}^\infty \vert\varphi(2^{-k}\Delta_g) f \vert^{2}\right)^\frac{1}{2}}_{L^p},\qquad f\in L^p(M).
	\end{align*}
	Hence, we get $\eqref{LittlewoodPaleyEquHilbert}$ by Fubini and $\eqref{LittlewoodPaleyEstLp}$
	 by Minkowski's inequality.
\end{proof}

The previous Lemma indicates the importance of estimating operators of the form $\Loc$ for $h\in (0,1].$ In the next Lemma, we state how they act in $L^p$-spaces and Sobolev spaces. Note that these kind of estimates are usually called \emph{Bernstein inequalities.}

\begin{Lemma}\label{Bernstein}
		\begin{enumerate}
		\item[a)] 	Let us assume that $1\le q\le r\le \infty.$ Then for any $\varphi\in C_c^\infty(\R)$, there is $C>0$ such that 
		\begin{align*}
		\norm{\Loc}_{ L^r(M)}\le C h^{d(\frac{1}{r}-\frac{1}{q})}\norm{\Loc u}_{ L^{q}}, \qquad u\in L^{q}(M), \;\;\; h\in(0,1].
		\end{align*}
		%	Especially, we get
		%	\begin{align*}
		%		\norm{\Loc}_{H^1\to L^2}\le \tilde{C} h.
		%	\end{align*}
		%	using $\Heins\hookrightarrow L^6(M).$
		\item[b)] 	Let us assume that  $p\in (1,\infty)$ and $s\ge 0.$ Then, for every $\varphi\in \realTestNoNull,$ there is $C>0$ such that 
		\begin{align*}
		\norm{\Loc u}_{ L^p}\le C h^s \norm{\Loc u}_{ H^{s,p}},\qquad u\in H^{s,p}(M),\;\; h\in(0,1].
		\end{align*}
	\end{enumerate}
\end{Lemma}

\begin{proof}
	\emph{ad a):} See \cite{Burq}, Corollary 2.2.\\	
		\emph{ad b):}
		Throughout this proof, we w.l.o.g.~assume $s>0$.
%		Let us fix  $\varphi\in \realTestNoNull$ and $s\ge 0.$
		Moreover, we take $\tilde{\varphi}\in \realTestNoNull$ with $\tilde{\varphi}=1$ on $\supp(\varphi)$ and define
		\begin{align*}
		f_h: [0,\infty) \to \R,\qquad f_h(t):=t^{-\frac{s}{2}}\tilde{\varphi}(-h^2 t)
		\end{align*}
		for $h\in (0,1].$ Then, we have $\varphi(-h^2t)=f_h(t)t^{\frac{s}{2}}\varphi(-h^2t)$ for all $t\in [0,\infty)$ and $h\in (0,1].$ Furthermore, we obtain that $f_h$ satisfies the Mihlin condition
		\begin{align*}
		\sup_{t\ge 0}\vert t^k f_h^{(k)}(t)\vert \lesssim h^s,\qquad k\in \N_0,\quad  h\in (0,1].
%		,\qquad \norm{f}_{L^\infty(0,\infty)}\lesssim h^s.
		\end{align*}
 Fact 2.20 in \cite{Uhl} and
%		Theorem 3.2 in \cite{OuhabazGaussianBounds}
%		(together with Lemmas 3.3 and 7.5 in \cite{OuhabazGaussianBounds})
			the Spectral Multiplier Theorem 7.6 in \cite{OuhabazGaussianBounds}
		hence imply
		\begin{align*}
			\norm{f_h(-\Delta_g)}_{\mathcal{L}(L^1,L^{1,\infty})}\lesssim h^s,\qquad  h\in (0,1].
		\end{align*}
		Since we also have 	
				\begin{align*}
		\norm{f_h(-\Delta_g)}_{\mathcal{L}(L^2)}\le \sup_{t\ge 0} \vert f_h(t)\vert \lesssim h^s, \qquad  h\in (0,1],
		\end{align*}
		by the Borel functional calculus for selfadjoint operators, the Marcinkiewitz Interpolation Theorem, see \cite{GrafakosClassical}, Theorem 1.3.2, yields
		\begin{align*}
		\norm{f_h(-\Delta_g)}_{\mathcal{L}(L^p)}\lesssim h^s, \qquad  h\in (0,1],
		\end{align*}		
		for $p\in (1,2].$ Since $f_h(-\Delta_g)$ is selfadjoint on $L^2(M)$, we obtain for $p\in (2,\infty)$
		\begin{align*}
		\norm{f_h(-\Delta_g)}_{\mathcal{L}(L^p)}&=\sup_{u\in L^p\cap L^2: \norm{u}_{L^p}\le 1} \sup_{v\in L^{p'}\cap L^2: \norm{v}_{L^{p'}}\le 1}\left\vert \skpLzwei{f_h(-\Delta_g)u}{v}\right\vert \\
		&=\sup_{v\in L^{p'}\cap L^2: \norm{v}_{L^{p'}}\le 1} \sup_{u\in L^p\cap L^2: \norm{u}_{L^p}\le 1}  \left\vert \skpLzwei{u}{f_h(-\Delta_g)v}\right\vert\\
		&=\norm{f_h(-\Delta_g)}_{\mathcal{L}(L^{p'})}
		\lesssim h^s, \qquad  h\in (0,1].
		\end{align*}
For every $p\in (1,\infty)$, we therefore get
	\begin{align*}
\norm{\Loc u}_{ L^p}&=\norm{f_h(-\Delta_g) \left(-\Delta_g\right)^{\frac{s}{2}}\Loc u}_{ L^p}
\lesssim  h^s \norm{\left(-\Delta_g\right)^{\frac{s}{2}}\Loc u}_{ L^p}\\
&\lesssim 		  h^s \norm{\Loc u}_{ H^{s,p}},\qquad u\in H^{s,p}(M).
\end{align*}		
This completes the proof of Lemma \ref{Bernstein}.

\end{proof}

%\begin{Lemma}\label{BernsteinSobolev}
%		Let $\varphi\in \realTestNoNull$ and $s\ge 0.$ Then, there is $C>0$ such that for  $h\in(0,1]$
%	\begin{align*}
%		\norm{\Loc f}_{ L^2}\le C h^s \norm{\Loc f}_{ H^s},\qquad f\in H^s(M).
%	\end{align*}
%\end{Lemma}
%
%\begin{proof}
%	Take $\varphi\in \realTestNoNull$ with $\tilde{\varphi}=1$ on $\supp(\varphi).$
%	We have $\varphi(h^2\lambda)=\lambda^{-\frac{s}{2}}\tilde{\varphi}(-h^2\lambda)\lambda^{\frac{s}{2}}\varphi(-h^2\lambda)$ for $\lambda\in\R$ and $\sup_{\lambda\in\R}\vert \lambda^{-\frac{s}{2}}\tilde{\varphi}(-h^2\lambda)\vert\lesssim h^s.$ Hence, the functional calculus yields
%	\begin{align*}
%		\norm{\Loc f}_{L^2}\lesssim h^s \norm{(-\Delta_g)^\frac{s}{2}\Loc f}_{L^2}\le h^s \norm{\Loc f}_{H^s},\qquad f\in H^s(M).
%	\end{align*}
%\end{proof}

%\begin{Lemma}\label{BernsteinSobolevLp}
%	Let $\varphi\in \realTestNoNull,$ $p\in (1,\infty)$ and $s\ge 0.$ Then, there is $C>0$ such that for  $h\in(0,1]$
%	\begin{align*}
%	\norm{\Loc f}_{ L^p}\le C h^s \norm{\Loc f}_{ H^{s,p}},\qquad f\in H^{s,p}(M).
%	\end{align*}
%\end{Lemma}

In the following Lemmata, we quote the spectrally localized Strichartz estimates from \cite{Burq}, which are a consequence of \cite{keelTao}. In this paper, Keel and Tao solved the endpoint case needed for our application in the proof of Proposition $\ref{controlHighNorms}.$

\begin{Lemma}\label{homogenousStrichartzLemma}
	Let $M$ be a compact riemannian manifold of dimension $d\ge 1$ and $p,q\in [2,\infty]$ with
		\begin{align*}
		\frac{2}{q}+\frac{d}{p}=\frac{d}{2},\qquad (q,p,d)\neq(2,\infty,2).
		\end{align*}
	%	\begin{enumerate}
	%		\item[i)] Let $M$ be a compact riemannian manifold of dimension $d\ge 1$ and $\Delta$ be the Laplace-Beltrami operator.
	%		\item[ii)] Let $M$ be a bounded polygonal domain in $\R^2$ and $\Delta$ be the Dirichlet or Neumann Laplacian.
	%	\end{enumerate}	
	Then, for any $\varphi\in C_c^\infty(\R)$, there is $\beta>0$ and $C>0$ such that for $h\in(0,1]$ and any interval $J$ of length $\vert J\vert\le \beta h$
	\begin{align}\label{homogenousStrichartzEstimateOriginal}
	\norm{t\mapsto e^{\im  t\Delta_g}\Loc x}_{L^q(J,L^p)}\le C \norm{x}_{L^2},\qquad x\in \Lzwei.
	\end{align}
\end{Lemma}

\begin{proof}
	See \cite{Burq}, Proposition 2.9. The result follows from the dispersive estimate for the Schr\"odinger group from \cite{Burq}, Lemma 2.5, and an application of Keel-Tao's Theorem (\cite{keelTao}) with $U(t)=\group \LocTilde \mathbf{1}_{J}(t)$ for some $\tilde{\varphi}\in C_c^\infty(\R)$ with $\tilde{\varphi}=1$ on $\operatorname{supp}(\varphi).$
\end{proof}

A similar result also holds for convolutions with the Schr\"odinger group.

\begin{Lemma}\label{inhomogenousLocalizedStrichartz}
	Let $M$ be a compact riemannian manifold of dimension $d\ge 1$ and $p_1,p_2,q_1,q_2 \in [2,\infty]$ with
	\begin{align*}
		\frac{2}{q_i}+\frac{d}{p_i}=\frac{d}{2},\qquad (q_i,p_i,d)\neq(2,\infty,2).
	\end{align*}
	For any $\varphi\in C_c^\infty(\R)$, there is $\beta>0$ and $C>0$ such that for $h\in(0,1]$ and any interval $J$ of length $\vert J\vert\le \frac{\beta h}{2}$
	\begin{align*}
		\bigNorm{t\mapsto \int_{-\infty}^{t} \groupTS\Loc f(s)\df s}_{L^{q_1}(J,L^{p_1})} \le C \norm{\Loc f}_{L^{q_2'}(J,L^{p_2'})}
	\end{align*}
\end{Lemma}

\begin{proof}
	See \cite{Burq}, Lemma 3.4.
\end{proof}

To prepare the next Lemma, we recall the following notation.

\begin{Notation}
	Let $E$ be a separable Banach space, $p\in [1,\infty),$ $J\subset [0,\infty)$ an interval and $\left(\Omega, \F, \Prob,\Filtration \right)$ a filtered probability space. By $\mathcal{M}^p(J,X),$ we denote the space of $\Filtration$-progressively measurable $E$-valued processes $\xi: J\times \Omega \to E$ with $\norm{\xi}_{L^p(J\times \Omega,E)}<\infty.$
\end{Notation}

Adapting the proof of Theorem 3.10 in \cite{BrzezniakStrichartz} to the present situation, we obtain a spectrally localized stochastic Strichartz estimate for stochastic convolutions with the Schr\"odinger group.
\begin{Lemma}\label{stochasticStrichartzLemma}
	Let $\varphi, \tilde{\varphi}\in \realTestNoNull$ with $\tilde{\varphi}=1$ on $\supp(\varphi).$ Choose $\beta>0$ as in Lemma $\ref{homogenousStrichartzLemma}.$ Let $h\in (0,1]$ and $J\subset [0,T]$ be an interval of length $\vert J\vert\le \beta h$ and $\chi_h\in \realTest$ with %$\supp(\chi_h)\subset J+[-\frac{\beta h}{2},\frac{\beta h}{2}]$ and %$\chi_h=1$ on $J.$
	$\supp(\chi_h)\subset J.$
	For $B\in \mathcal{M}^2(J,\HS(Y,L^2)),$ we set
	\begin{align*}
		G(t):=\int_{0}^t \groupTS \chi_h(s) \Loc B(s)\df W(s),\qquad t\in J.
	\end{align*}
	Then,
	\begin{align*}
		\norm{G}_{L^2(\Omega,\JendpointSpace)}\lesssim \norm{\LocTilde B}_{L^2(\Omega,L^2(J,\HS(Y,L^2)))}.
	\end{align*}	
\end{Lemma}

\begin{proof}
	We abbreviate
	\begin{align*}
		F(t,s):=\mathbf{1}_{\{s\le t\}} \groupTS \chi_h(s) \Loc B(s),\qquad t,s\in J,
	\end{align*}
	and use the Burkholder-Davis-Gundy-inequality in the martingale type 2 Banach space $\JendpointSpace$ (see for example \cite{BrzezniakConvolutions}) to estimate
	\begin{align}\label{BurkholderStochStrichartz}
		\norm{G}_{L^2(\Omega,\JendpointSpace)}^2=\E \bigNorm{\int_J F(\cdot,s) \df W(s)}_{\JendpointSpace}^2
		\lesssim \E \int_J \norm{F(\cdot,s)}_\JendpointSpaceGamma^2 \df s
	\end{align}
	Writing out the definition of $\gamma(Y,L^2(J,L^6))$ and using  $\Loc= \Loc\LocTilde,$ we get
	\begin{align*}
		\norm{F(\cdot,s)}_\JendpointSpaceGamma^2=&\Etilde \bigNorm{t\mapsto\sumM \gamma_m \mathbf{1}_{\{s\le t\}} \groupTS \chi_h(s) \Loc B(s)f_m}_\JendpointSpace^2\\
		=&\Etilde \bigNorm{t\mapsto \sumM \group \Loc \left[\gamma_m  \groupMinusS \chi_h(s) \LocTilde B(s)f_m\right]}_{L^2(J_{\ge s},L^6)}^2,
	\end{align*}
	where $\left(\gamma_m\right)_{m\in\N}$ is a sequence of i.i.d. $\mathcal{N}(0,1)$-Gaussians on some probability space $\tilde{\Omega}.$
	By Lemma $\ref{homogenousStrichartzLemma},$ the operator $e^{\im \cdot \Delta_g}\Loc$ is bounded from $\Lzwei$ to $\JendpointSpace.$ Hence, we can take it out of the sum and obtain
	\begin{align*}
				\norm{F(\cdot,s)}_\JendpointSpaceGamma^2
				\lesssim&\Etilde \bigNorm{ \sumM \gamma_m \ \groupMinusS \chi_h(s) \LocTilde B(s)f_m}_{L^2}^2\\
				=&\norm{  \groupMinusS \chi_h(s) \LocTilde B(s)}_{\gamma(Y,L^2)}^2
				\eqsim \norm{  \groupMinusS \chi_h(s) \LocTilde B(s)}_{\HS(Y,L^2)}^2\\
				\lesssim& \norm{ \chi_h(s) \LocTilde B(s)}_{\HS(Y,L^2)}^2.
	\end{align*}
	Finally, inserting the last estimate in $\eqref{BurkholderStochStrichartz}$  yields
	\begin{align*}
		\norm{G}_{L^2(\Omega,\JendpointSpace)}^2\lesssim \E \int_J \norm{ \chi_h(s) \LocTilde \LocTilde B(s)}_{\HS(Y,L^2)}^2 \df s\lesssim \norm{\LocTilde B}_{L^2(\Omega,L^2(J,\HS(Y,L^2)))}^2.
	\end{align*}
	The proof of Lemma \ref{stochasticStrichartzLemma} is thus completed.
\end{proof}

\section{Uniqueness}

In the following section, we will prove the pathwise uniqueness of solutions of \eqref{ProblemStratonovich}. A key ingredient for this result is an $L^2_tL^p_x$-estimate for solutions for arbitrary large $p$ with moderate growth of the bound in $p.$

\begin{Prop}\label{controlHighNorms}
	Let $d=3$ and $\alpha\in (1,3].$
	Let $\left(\Omega,\F,\Prob,W,\Filtration,u\right)$ be a martingale solution of $\eqref{ProblemStratonovich}.$ Then, there is a measurable set $\Omega_1\subset \Omega$ with $\Prob(\Omega_1)=1$ such that for all $\omega\in \Omega_1,$ $p\in [6,\infty)$ and intervals $J\subset [0,T],$ we have  $u(\cdot,\omega)\in L^2(J;L^{p}(M))$  with 
	\begin{align*}
	\norm{u(\cdot,\omega)}_{L^2(J,L^p)}\lesssim_\omega\, 1+\left(\vert J\vert p\right)^\frac{1}{2}.
	\end{align*}	
\end{Prop}

%\begin{Prop}\label{controlHighNorms}
%	Let $d=3$ and $\alpha\in (1,3].$ Let $\left(\Omega,\F,\Prob,W,\Filtration,u\right)$ be a martingale solution of $\eqref{ProblemStratonovich}$ and $J\subset [0,T]$ be an interval. Then, we have  $u\in L^q(J;L^{p}(M))$ $\tildeProb$-almost surely with 
%	\begin{align*}
%	\norm{u}_{L^q(J,L^p)}\lesssim 1+\left(\vert J\vert p\right)^\frac{1}{2}\qquad \text{a.s.}
%	\end{align*}
%	for all $p\in [6,\infty)$ and $q\in [1,2].$
%\end{Prop}
We remark that this estimate of $L^p$-norms is a substitute for the $L^\infty$-bound for solutions in the $2D$-setting, see \cite{ExistencePaper}, and complements the inequality
\begin{align*}
\norm{u}_{L^2(J,L^p)}\lesssim \vert J\vert^\frac{1}{2}\norm{u}_{L^\infty(J,H^1)}<\infty \qquad \text{a.s.}
\end{align*}
for $p\in [1,6],$ which we get from  Sobolev's embedding and the energy estimate for martingale solutions.
Before we start with the proof, we introduce an equidistant partition of the time interval.

\begin{Notation}\label{notationZerlegung}
	Let $I=[a,b]$ with $0<a<b<\infty.$ For $\rho>0$ and $N:=\lfloor \frac{b-a}{\rho}\rfloor,$ i.e. $N=\max\{n\in \N\colon n\le \frac{b-a}{\rho}\}$, the family $\left(I_j\right)_{j=0}^N$ defined by
	\begin{align*}
	I_j:=& \left[a+j\rho,a+(j+1)\rho\right],\quad j\in \{0,\dots N-1\},\\
	I_N:=& \left[a+N\rho,b\right]
	\end{align*}
	is called \emph{$\rho$-partition of $I$}. Observe 
	\begin{align*}
	\vert I_j\vert\le \rho,\quad j=0,\dots,N,\qquad I=\bigcup_{j=0}^N I_j,\qquad I_j^\circ\cap I_k^\circ=\emptyset,\quad j\neq k. 
	\end{align*}
\end{Notation}

\begin{proof}[Proof of Proposition \ref{controlHighNorms}]
	\emph{Step 1.}
	%	Decompose $[a,b]$ in the disjoint union 
	%	\begin{align*}
	%		[a,b]=[a,\tilde{a}]\cup \bigcup_{k=1}^N J_k \cup [\tilde{b},b]
	%	\end{align*}
	%	with closed intervals $J_k,$ such that $\vert J_k\vert=\frac{\alpha h}{4}$ and $J_k':=J_k+[-\frac{\alpha h}{8},\frac{\alpha h}{8}]\subset [a,b]$
	%	\begin{align*}
	%		0\le \tilde{a}-a\le \alpha h,\quad 0\le b-\tilde{b}\le \alpha h,\quad N=\left\lfloor\frac{4(b-a)}{\alpha h}\right\rfloor
	%	\end{align*}	
	We choose $\beta>0$ and $h\in (0,1]$ as in Lemma $\ref{homogenousStrichartzLemma}$ and take  a $\frac{\beta h}{4}$-partition $\left(I_j\right)_{j=0}^{N_T}$  of $[0,T]$ in the sense of Notation \ref{notationZerlegung}. Furthermore, we define a cover $\left(I_j'\right)_{j=0}^{N_T}$ of $\left(I_j\right)_{j=0}^{N_T}$ by
	\begin{align*}
	I_j':=\left(I_j+\left[-\frac{\beta h}{8},\frac{\beta h}{8}\right] \right)\cap [0,T],
	\qquad m_j:=\frac{j \beta h}{4}+\frac{\beta h}{8},\qquad j=0,\dots,N_T,
	\end{align*}
	and a sequence $\left(\chi_{I_j}\right)_{j=0}^{N_T}\subset C_c^\infty([0,\infty))$ by $\chi_{I_j}:=\chi\left((\beta h)^{-1}(\cdot-m_j)\right)$ for some $\chi\in \realTest$ with $\chi=1$ on $[-\frac{1}{8},\frac{1}{8}]$ and $\supp(\chi)\subset [-\frac{1}{4},\frac{1}{4}].$ Then, we have
	\begin{align}\label{ZdE}
	\chi_{I_j}=1 \quad \text{on  } I_j,\qquad \supp(\chi_{I_j}) \subset I_j', \quad \norm{\chi_{I_j}'}_{L^\infty(\R)}\le (\beta h)^{-1} \norm{\chi'}_{L^\infty(\R)}.
	\end{align} 		
	%	Let $J=[c-\frac{\alpha h}{8},c+\frac{\alpha h}{8}]\subset [0,\infty)$ be an interval of length $\vert J\vert= \frac{\alpha h}{4}$ and $\chi_h\in \realTest$ with %$\supp(\chi_h)\subset J+[-\frac{\alpha h}{2},\frac{\alpha h}{2}]$ and %$\chi_h=1$ on $J.$ 
	%	$\supp(\chi_h)\subset J.$ 
	We fix $\varphi, \tilde{\varphi}\in \realTestNoNull$ with $\tilde{\varphi}=1$ on $\supp(\varphi).$ 
	In order to localize the solution $u$ spectrally and in time, we set 
	\begin{align*}
	v_{I_j}(t)=\chi_{I_j}(t) \Loc u(t),\qquad j=0,\dots,N_T,
	\end{align*}
%	for an arbitrary $\varphi\in \realTestNoNull.$
	and apply It\^o's formula to $\varPhi_j \in C^{1,2}(I_j'\times \Hminusdrei,\Hminuseins)$ defined by
		\begin{align*}
		\varPhi_j(s,x)=\groupTS\chi_{I_j}(s)\Loc x,\qquad s\in I_j',\quad x\in \Hminusdrei,
		\end{align*}
		to get the representation of $v_{I_j}$ in the mild form 
			\begin{align}\label{mildFormMiddleInterval}
			v_{I_j}(t)=&\int_{\min I_j'}^t \left[-\im \Delta_g \groupTS \chi_{I_j}(s)\Loc u(s)+\groupTS \chi_{I_j}'(s)\Loc u(s)\right]\df s\nonumber\\
			&+\int_{\min I_j'}^t \groupTS \chi_{I_j}(s) \Loc \left[\im \Delta_g u(s)-\im \lambda \vert u(s)\vert^{\alpha-1} u(s)+\mu(u(s))\right] \df s\nonumber\\
			&
			-\im \int_{\min I_j'}^t \groupTS \chi_{I_j}(s) \Loc B u(s)\df W(s)\nonumber\\
			=&\int_{\min I_j'}^t \groupTS \chi_{I_j}'(s)\Loc u(s)\df s\nonumber\\
			&+\int_{\min I_j'}^t \groupTS \chi_{I_j}(s) \Loc \left[-\im\lambda \vert u(s)\vert^{\alpha-1} u(s)+\mu(u(s))\right] \df s\nonumber\\
			&
			-\im \int_{\min I_j'}^t \groupTS \chi_{I_j}(s) \Loc B u(s)\df W(s)
			\end{align}
			for $j=1,\dots, N_T$ in $\Hminusdrei$ almost surely for $t\in I_j'.$ Because of the regularity of each term (recall $\alpha\le 3$), this identity also holds in $\Lzwei.$ 
%
%
%
%
%
%
%
%	We recall that $u$ has the representation
%	\begin{align*}
%	u(t)=&u_0+\int_0^t \left\{\im \Delta_g u(s)-\im\lambda \vert u(s)\vert^{\alpha-1} u(s)
%	+ \mu(u(s))\right\}\df s-\im\int_0^t B(u(s))\df W(s)
%	\end{align*}
%	in $H^{-1}(M)$ almost surely for all $t\in [0,T]$ and from the It\^o formula as well as $\chi_{I_j}=0$ on $[0,\min I_j'],$ we infer
%	\begin{align}\label{strongFormMiddleInterval}
%	v_{I_j}(t)=&\int_{\min I_j'}^t \Big\{ \im \Delta_g v_{I_j}(s)+\chi_{I_j}'(s)\Loc u(s)+\chi_{I_j}(s) \Loc \left[-\im\lambda \vert u(s)\vert^{\alpha-1} u(s)+\mu(u(s))\right]\Big\} \df s\nonumber\\
%	&
%	-\im \int_{\min I_j'}^t  \chi_{I_j}(s) \Loc B u(s)\df W(s)
%	\end{align}
%	in $H^{-1}(M)$ almost surely for all $t\in I_j'.$
%
%	
%	
%	
%	
%	 Next, we employ Lemma \ref{StrongMildEquivalence} with $X=H^{-1}(M)$ and $A f=-\Delta_g f$ for $f\in \Heins=:D(A)$ to rewrite \eqref{strongFormMiddleInterval} in the mild form 
%	\begin{align}\label{mildFormMiddleInterval}
%	v_{I_j}(t)
%	=&\int_{\min I_j'}^t \groupTS \chi_{I_j}'(s)\Loc u(s)\df s\nonumber\\
%	&+\int_{\min I_j'}^t \groupTS \chi_{I_j}(s) \Loc \left[-\im\lambda \vert u(s)\vert^{\alpha-1} u(s)+\mu(u(s))\right] \df s\nonumber\\
%	&
%	-\im \int_{\min I_j'}^t \groupTS \chi_{I_j}(s) \Loc B u(s)\df W(s)
%	\end{align}
%	for $j=1,\dots, N_T$ in $\Hminuseins$ almost surely for $t\in I_j'.$ Because of  $\alpha\le 3$, each term is so regular  that this identity also holds in $\Lzwei.$ 
	Analogously, we get 
	\begin{align}\label{mildFormStartInterval}
	v_{I_0}(t)=&\group v_{I_0}(0)+\int_{0}^t \groupTS \chi_{I_0}'(s)\Loc u(s)\df s\nonumber\\
	&+\int_{0}^t \groupTS \chi_{I_0}(s) \Loc \left[-\im\lambda \vert u(s)\vert^{\alpha-1} u(s)+\mu(u(s))\right] \df s\nonumber\\
	&
	-\im \int_{0}^t \groupTS \chi_{I_0}(s) \Loc B u(s)\df W(s)
	\end{align}
	in $\Lzwei$ almost surely for $t\in I_0'.$
	We abbreviate 
	\begin{align*}
	G_{I_j}(t):=\int_{\min I_j'}^t \groupTS \chi_{I_j}(s) \Loc B u(s)\df W(s)
	\end{align*}
	for $0\le t \in [0,T].$ We use the stochastic Strichartz estimate from Lemma $\ref{stochasticStrichartzLemma},$ the properties of $\left(I_j\right)_{j=0}^{N_T}$ and $\left(I_j'\right)_{j=0}^{N_T}$ and Lemma $\ref{Bernstein}$ b) to estimate
	\begin{align*}
	\E \sum_{j=0}^{{N_T}} \norm{G_{I_j}}_\IjStrichendpointSpace^2
	\lesssim\,& \E \sum_{j=0}^{{N_T}} \int_{I_j'}\norm{\LocTilde B(u(s))}_{\HS(Y,L^2)}^2 \df s\\
	\le& 2 \E \sum_{j=0}^{{N_T}} \int_{I_j}\norm{\LocTilde B(u(s))}_{\HS(Y,L^2)}^2 \df s\\
	=&  2\E  \int_0^T\norm{\LocTilde B(u(s))}_{\HS(Y,L^2)}^2 \df s\\
	\lesssim\,& h^2 \E \int_0^T \norm{\LocTilde B(u(s))}_{\HS(Y,H^1)}^2 \df s.
	\end{align*}
	Since $\LocTilde$ is a bounded operator from $\Heins$ to $\Heins$ and $B$ is bounded from $\Heins$ to $\HS(Y,\Heins)$ by Assumption $\ref{stochasticAssumptions}$, we conclude
	\begin{align*}
	\E \sum_{j=0}^{{N_T}} \norm{G_{I_j}}_\IjStrichendpointSpace^2\lesssim\,& h^2 \E \int_0^T \norm{u(s)}_{H^1}^2 \df s.
	\end{align*}
	Hence, there is $C=C(\omega)$ with $C<\infty$ almost surely such that
	\begin{align}\label{stochasticSumEstimateComplete}
	\sum_{j=0}^{{N_T}} \norm{G_{I_j}}_\IjStrichendpointSpace^2\le h^2 C \qquad \text{a.s.}
	\end{align}
	
	\emph{Step 2.} We fix a path $\omega\in \Omega_h,$ where $\Omega_h$ is the intersection of the full measure sets from $\eqref{mildFormMiddleInterval},$ $\eqref{mildFormStartInterval},$ $\eqref{stochasticSumEstimateComplete}$ and $u_j\in L^\infty(0,T;H^1(M))$ almost surely. In the rest of the argument, we skip the dependence of $\omega$ to keep the notation simple. 
	Let us pick those intervals $J_0,\dots,J_N$ from the partition $(I_j)_{j=0}^{N_T}$ which cover the given interval $J.$ The associated intervals in $(I_j')_{j=0}^N$ will be denoted by $J_0',\dots, J_N'.$ 
	From \eqref{stochasticSumEstimateComplete}, we infer 
	\begin{align}\label{stochasticSumEstimate}
	\sum_{j=0}^{{N}} \norm{G_{J_j}}_\JjStrichendpointSpace^2\le h^2 C.
	\end{align}	
	Applying the homogeneous and inhomogenous Strichartz estimates from  Lemma $\ref{homogenousStrichartzLemma}$ and  $\ref{inhomogenousLocalizedStrichartz}$ in $\eqref{mildFormMiddleInterval}$ and in $\eqref{mildFormStartInterval},$ we
	obtain	
	\begin{align}\label{IjEstimateStart}
	\norm{v_{J_j}}_{L^2(J_j,L^6)}\le \norm{v_{J_j}}_{L^2(J_j',L^6)}\lesssim\,& \norm{\chi_{J_j}'\Loc u}_{L^1(J_j',L^2)}
	+\norm{\chi_{J_j}\Loc \vert u \vert^{\alpha-1} u}_{L^2(J_j',L^{\frac{6}{5}})}\nonumber\\
	&+\norm{\chi_{J_j}\Loc \mu(u)}_{L^1(J_j',L^2)}		
	+\norm{G_{J_j}}_{L^2(J_j',L^6)}
	\end{align}
	for $j=1,\dots,N$ and 
	\begin{align}\label{initialEstimateStart}
	\norm{v_{J_0}}_{L^2(J_0,L^6)}\le \norm{v_{J_0}}_{L^2(J_0',L^6)}\lesssim\,& \norm{v_{J_0}(\min J_0')}_{L^2}+ \norm{\chi_{J_0}'\Loc u}_{L^1(J_0',L^2)}
	\nonumber\\
	&+\norm{\chi_{J_0}\Loc \vert u \vert^{\alpha-1} u}_{L^2(J_0',L^{\frac{6}{5}})}+\norm{\chi_{J_0}\Loc \mu(u)}_{L^1(J_0',L^2)}
	\nonumber\\
	&+\norm{G_{J_0}}_{L^2(J_0',L^6)}.
	\end{align}	
	Note that $v_{J_0}(\min J_0')=0$ if $I_0\neq J_0.$ 
	Next, we estimate the terms on the right hand side of $\eqref{IjEstimateStart}$ and $\eqref{initialEstimateStart}.$
	By $\eqref{ZdE},$ Lemma $\ref{Bernstein}$ b) and H\"older's inequality, we get
	\begin{align*}
	\norm{\chi_{J_j}'\Loc u}_{L^1(J_j',L^2)}
	&\lesssim\,  h^{-1} \norm{\Loc u}_{L^1(J_j',L^2)} 
	\lesssim\,   \norm{\Loc u}_{L^1(J_j',H^1)}\nonumber\\
	&\lesssim\, h^\frac{1}{2}  \norm{\Loc u}_{L^2(J_j',H^1)}.
	\end{align*}	
	H\"older's inequality with $\vert J_j'\vert \lesssim h,$ Lemma $\ref{Bernstein}$ b) and the boundedness of the operators $\Loc$ and $\mu$ in $\Heins$ yield
	\begin{align*}
	\norm{\chi_{J_j}\Loc \mu(u)}_{L^1(J_j',L^2)}
	\lesssim\,& h \norm{\chi_{J_j}\Loc \mu(u)}_{L^\infty(J_j',L^2)}
	\le  h \norm{\Loc \mu(u)}_{L^\infty(0,T;L^2)}\\
	\lesssim\,& h^2 \norm{\Loc \mu(u)}_{L^\infty(0,T;H^1)}
	\lesssim\,  h^2 \norm{u}_{L^\infty(0,T;H^1)}.
	\end{align*}
	We apply Lemma $\ref{LemmaEstimateNonlinearity}$ with $r'=\frac{6}{\alpha+2}\ge \frac{6}{5}$ and $q=6$ and obtain the estimate 
	\begin{align*}
	\norm{\vert v \vert^{\alpha-1} v}_{H^{1,\frac{6}{5}}}
	\lesssim\, \norm{\vert v \vert^{\alpha-1} v}_{H^{1,\frac{6}{\alpha+2}}}
	\lesssim\, \norm{v}_{H^1}^{\alpha},\qquad v\in\Heins,
	\end{align*}
	where we used $\alpha\le 3.$
	Together with H\"older's inequality, Lemma $\ref{Bernstein}$ b) and the boundedness of $\Loc$, this implies
	\begin{align*}
	\norm{\chi_{J_j}\Loc \vert u \vert^{\alpha-1} u}_{L^2(J_j',L^{\frac{6}{5}})}
	&\lesssim\, h^\frac{1}{2}\norm{\Loc \vert u \vert^{\alpha-1} u}_{L^\infty(0,T;L^{\frac{6}{5}})}\\
	&\lesssim\, h^\frac{3}{2}\norm{\Loc \vert u \vert^{\alpha-1} u}_{L^\infty(0,T;H^{1,\frac{6}{5}})}\\
	&
	\lesssim\, h^\frac{3}{2}\norm{ \vert u \vert^{\alpha-1} u}_{L^\infty(0,T;H^{1,\frac{6}{5}})}\lesssim\, h^\frac{3}{2}\norm{ u}_{L^\infty(0,T;H^{1})}^\alpha.
	\end{align*}		
	Inserting the last three estimates in $\eqref{IjEstimateStart}$ and $\eqref{initialEstimateStart}$ yields
	\begin{align}\label{IjEstimate}
	\norm{v_{J_j}}_{L^2(J_j,L^6)}
	\lesssim\,& h^\frac{1}{2} \norm{\Loc u}_{L^2(J_j',H^1)}
	+ h^{\frac{3}{2}} \norm{u}_{L^\infty(0,T;H^1)}^\alpha\nonumber\\
	&+h^2 \norm{u}_{L^\infty(0,T;H^1)}
	+\norm{G_{J_j}}_{L^2(J_j',L^6)},
	\end{align}	
	\begin{align}\label{initialEstimate}
	\norm{v_{J_0}}_{L^2(J_0,L^6)}
	\lesssim\,&  h\norm{\Loc u(\min J_0')}_{H^1}+
	h^\frac{1}{2} \norm{\Loc u}_{L^2(J_0',H^1)}
	+ h^{\frac{3}{2}} \norm{u}_{L^\infty(0,T;H^1)}^\alpha\nonumber\\
	&+h^2 \norm{u}_{L^\infty(0,T;H^1)}
	+\norm{G_{J_0}}_{L^2(J_0',L^6)}.
	\end{align}	
	 We square the estimates $\eqref{IjEstimate}$ and $\eqref{initialEstimate}$ and sum them up.  Using $\chi_{J_j}=1$ on $J_j,$ $\eqref{stochasticSumEstimate}$  and $N\le N_T=\left\lfloor \frac{4 T}{\beta h}\right\rfloor,$ we conclude 
	\begin{align}\label{temporalSum}
	\norm{\Loc u}_{L^{2}(J,L^6)}^{2}\le&\sum_{j=0}^N  \norm{\chi_{J_j}\Loc u}_{L^{2}(J_j,L^6)}^{2}=\sum_{j=0}^N  \norm{v_{J_j}}_{L^{2}(J_j,L^6)}^{2}\nonumber\\
	\lesssim\,&h^{2}\norm{\Loc u(\min J_0')}_{H^1}^{2}\nonumber\\
	&+\sum_{j=0}^N \left[h \norm{\Loc u}_{L^2(J_j',H^1)}^{2}
	+ h^{3} \norm{u}_{L^\infty(0,T;H^1)}^{2\alpha }\right]\nonumber\\
	&+\sum_{j=0}^N \left[h^{4} \norm{u}_{L^\infty(0,T;H^1)}^{2}\right]
	+h^{2} C	\nonumber\\
	\lesssim\,& h^{2}\norm{\Loc u(\min J_0')}_{H^1}^{2}
	+h \sum_{j=0}^N  \norm{\Loc u}_{L^2(J_j',H^1)}^{2}
	\nonumber\\&+ h^{2} \norm{u}_{L^\infty(0,T;H^1)}^{2\alpha}
	+h^{3} \norm{u}_{L^\infty(0,T;H^1)}^{2}
	+h^{2} C.	
	\end{align}
	Below, we will use the notations
	\begin{align*}
	J_{N+1}:=\left(\bigcup_{j=0}^N J_j'\right)\setminus \left(\bigcup_{j=0}^N J_j\right),\qquad J^h:=\bigcup_{j=0}^{N+1} J_j.
	\end{align*}
	By
	\begin{align*}
	\sum_{j=0}^N  \norm{\Loc u}_{L^2(J_j',H^1)}^{2}&\le 2\sum_{j=0}^{N+1}  \norm{\Loc u}_{L^2(J_j,H^1)}^2=2 \norm{\Loc u}_{L^2(J^h,H^1)}^{2}
%	&\le 2 (N+1)^{1-\frac{q}{2}} \Big(\sum_{j=0}^{N+1}  \norm{\Loc u}_{L^2(J_j,H^1)}^2\Big)^{\frac{q}{2}} \\
%	&= 2 (N+1)^{1-\frac{q}{2}} \norm{\Loc u}_{L^2(J^h,H^1)}^{q} \\
%	&= h^{-1+\frac{q}{2}} \norm{\Loc u}_{L^2(J^h,H^1)}^{q},
	\end{align*}
	we obtain
	\begin{align*}
	\norm{\Loc u}_{L^{2}(J,L^6)}^{2}
	\lesssim\,& h^{2}\norm{\Loc u(\min J_0')}_{H^1}^{2}
	+h \norm{\Loc u}_{L^2(J^h,H^1)}^{2} \nonumber\\
	&+ h^{2} \norm{u}_{L^\infty(0,T;H^1)}^{2 \alpha}
	+h^{3} \norm{u}_{L^\infty(0,T;H^1)}^2
	+h^2 C.
	\end{align*}
	%	\textbf{TODO: Apply Lemma $\ref{BernsteinLq}$ to get the Lp-estimate and then sum over $h=2^{-k}$ with Littlewood Paley} 
	Let $p\ge 6.$ Then, Lemma $\ref{Bernstein}$ a) and $u\in L^\infty(0,T;\Heins)$ imply
	\begin{align}\label{localLpEstimate}
	\norm{\Loc u}_{L^{2}(J,L^p)}&\lesssim\, h^{3\left(\frac{1}{p}-\frac{1}{6}\right)}\norm{\Loc u}_{L^{2}(J,L^6)}\nonumber\\
	&\lesssim\, h^{\frac{3}{p}+\frac{1}{2}}\norm{\Loc u(\min J_0')}_{H^1}
	+h^{\frac{3}{p}} \norm{\Loc u}_{L^2(J^h,H^1)}\nonumber\\
	&\hspace{0.5cm}
	+ h^{\frac{3}{p}+\frac{1}{2}} \norm{u}_{L^\infty(0,T;H^1)}^\alpha
	+h^{\frac{3}{p}+1} \norm{u}_{L^\infty(0,T;H^1)}
	+h^{\frac{3}{p}+\frac{1}{2}} C	\nonumber\\
	&\lesssim\,
	h^{\frac{3}{p}+\frac{1}{2}}+h^{\frac{3}{p}} \norm{\Loc u}_{L^2(J^h,H^1)}
	+ h^{\frac{3}{p}+\frac{1}{2}} 
	+h^{\frac{3}{p}+1}. 
	\end{align}
	
	\emph{Step 3.} In the last step, we use  $\eqref{localLpEstimate}$ and Littlewood-Paley theory to derive the estimate stated in the Proposition.
	To this end, we set $h_k:=2^{-\frac{k}{2}}$ and $k_0:=\min \left\{k: \vert J\vert > \frac{\beta h_{k}}{4}\right\}.$
	Let us define $\Omega_1:=\bigcap_{k=1}^\infty \Omega_{h_k}$ and fix a path $\omega\in \Omega_1.$ We remark that we have $\Prob(\Omega_1)=1$ by the choice of $\Omega_h$ for each $h\in (0,1]$ from the previous step.  In the rest of the argument, we skip the dependence of $\omega$ to keep the notation simple.
	 Moreover, we choose $\psi \in \realTest,$ $\varphi\in \realTestNoNull$ such that
	\begin{align*}
	1=\psi(\lambda)u+\sum_{k=1}^\infty \varphi(2^{-k}\lambda) ,\qquad \lambda\in\R.
	\end{align*}
	Then, Lemma $\ref{LittlewoodPaley},$ the embedding $\ell^1(\N)\hookrightarrow \ell^2(\N)$ and $\eqref{localLpEstimate}$ imply 
	\begin{align}\label{finalEstimateHighNormsLemma}
	\norm{u}_{L^{2}(J,L^p)}
	\lesssim\, & \bigNorm{\left(\norm{\psi(\Delta_g)u}_{L^p}^2+\sum_{k=1}^\infty \norm{\LocK u}_{L^p}^2\right)^\frac{1}{2}}_{L^{2}(J)}\nonumber\\
	=\,& \left(\norm{\psi(\Delta_g)u}_{L^2(J,L^p)}^2+\sum_{k=1}^\infty \norm{\LocK u}_{L^2(J,L^p)}^2\right)^\frac{1}{2}\nonumber\\	
	\le\,& \norm{\psi(\Delta_g)u}_{L^2(J,L^p)}+\sum_{k=1}^\infty \norm{\LocK u}_{L^2(J,L^p)}\nonumber\\	
	\lesssim\,& \norm{\psi(\Delta_g)u}_{L^{2}(J,L^p)}+\sum_{k=1}^{k_0-1} \norm{\LocK u}_{L^{2}(J,L^p)}\nonumber\\
	&+\sum_{k=k_0}^\infty  
	2^{-\frac{3k}{2 p}} \norm{\LocK u}_{L^2(J^{h_k},H^1)}
	+\sum_{k=k_0}^\infty  \left[2^{-\frac{k}{2}\left(\frac{3}{p}+\frac{1}{2}\right)} 
	+2^{-\frac{k}{2}\left(\frac{3}{p}+1\right)} 
	+2^{-\frac{k}{2} \left(\frac{3}{p}+\frac{1}{2}\right)} \right]
	\nonumber\\
	\le& \norm{\psi(\Delta_g)u}_{L^{2}(J,L^p)}+\sum_{k=1}^{k_0-1} \norm{\LocK u}_{L^{2}(J,L^p)}\nonumber\\
	&+\sum_{k=k_0}^\infty  
	2^{-\frac{3k}{2 p}} \norm{\LocK u}_{L^2(J^{h_k},H^1)}+\sum_{k=k_0}^\infty  \left[2^{-\frac{k}{4}} 
	+2^{-\frac{k}{2}} 
	+2^{-\frac{k}{4} } \right]
	\nonumber\\
	\lesssim\,& \norm{\psi(\Delta_g)u}_{L^{2}(J,L^p)}+\sum_{k=1}^{k_0-1} \norm{\LocK u}_{L^{2}(J,L^p)}\nonumber\\
	&+ \left(\sum_{k=k_0}^\infty  
	2^{-\frac{3k}{p}} \right)^\frac{1}{2} 
	\left(\sum_{k=k_0}^\infty  
	\norm{\LocK u}_{L^2(J^{h_k},H^1)}^2\right)^\frac{1}{2}+1.
	%		\lesssim\,& 1+ \left(\frac{1}{1-2^{-\frac{6}{p}}}-1\right)^\frac{1}{2}\nonumber\\
	%		=&1+ (\vert J\vert p )^\frac{1}{2} \left(\frac{2^{-\frac{6}{p}}}{\vert J\vert p \left(1-2^{-\frac{6}{p}}\right)}\right)^\frac{1}{2}\lesssim\, 1+ (\vert J\vert p )^\frac{1}{2}
	\end{align}
	From Lemma $\ref{Bernstein}$ a) with $h=1,$ we conclude
	\begin{align}\label{inhomLittlewoodPaleyEstimate}
	\norm{\psi(\Delta_g)u}_{L^{2}(J,L^p)}
	\lesssim\, \norm{\psi(\Delta_g)u}_{L^{2}(J,L^2)}
	\lesssim\, \norm{u}_{L^{2}(J,L^2)}
	\lesssim\, 1.
	\end{align}
	From  Lemma $\ref{Bernstein}$ a) and the Sobolev embedding, we infer
	\begin{align*}
	\norm{\LocK u}_{L^{2}(J,L^p)}&\lesssim  2^{-k(\frac{3}{2 p}-\frac{1}{4})} \norm{\LocK u}_{L^{2}(J,L^6)}\\
	&\lesssim  2^{\frac{k}{4}} \norm{\LocK u}_{L^{2}(J,H^1)}
	\end{align*}
	for $k\in \left\{1,\dots,k_0-1\right\}.$ From the definition of $k_0,$ we have $\vert J\vert \eqsim 2^{-\frac{k_0}{2}}.$ Thus, we get  
	\begin{align}\label{MiddleFrequenciesLittlewoodPaleyEstimate}
	\sum_{k=1}^{k_0-1} \norm{\LocK u}_{L^{2}(J,L^p)}
	&\lesssim  \left(\sum_{k=1}^{k_0-1} 2^{\frac{k}{2}}\right)^\frac{1}{2} \left(\sum_{k=1}^{k_0-1}\norm{\LocK u}_{L^{2}(J,H^1)}^2\right)^\frac{1}{2}\nonumber\\
	&\lesssim  2^{\frac{k_0}{4}} \norm{u}_{L^{2}(J,H^1)}
	\lesssim  \vert J\vert^{-\frac{1}{2}} \vert J\vert^{\frac{1}{2}}\lesssim 1.
	\end{align}
	We proceed with the estimate of the sums over $k\ge k_0.$ The fact that we have $J^{h_{k+1}}\subset J^{h_{k}}$ for all $k\in\N,$ leads to 
	\begin{align*}
	\sum_{k=k_0}^\infty  
	\norm{\LocK u}_{L^2(J^{h_k},H^1)}^2
	&=\sum_{k: \vert J\vert > \frac{\beta h_k}{4}} \norm{\LocK u}_{L^2(J^{h_k},H^1)}^2\nonumber\\
	&\le \sum_{k: \vert J\vert > \frac{\beta h_k}{4}} \norm{\LocK u}_{L^2(J^{h_{k_0}},H^1)}^2\nonumber\\
	%			&\le\sum_{k=1}^\infty \norm{\LocK u}_{L^2(J^{h_{k_0}},H^1)}^2\nonumber\\
	&\lesssim \norm{u}_{L^2(J^{h_{k_0}},H^1)}^2\le \vert J^{h_{k_0}}\vert\,\norm{u}_{L^\infty(0,T;H^1)}^2.
	\end{align*}
	Using $\vert J^{h_{k_0}}\vert \le 3 \frac{\beta h_{k_0}}{4}+\vert J\vert\le 4\vert J\vert$
	and $u\in L^\infty(0,T;\Heins)$ almost surely, we obtain
	\begin{align}\label{intervalLengthEstimate}
	\sum_{k=k_0}^\infty 
	\norm{\LocK u}_{L^2(J^{h_k},H^1)}^2\lesssim \vert J\vert.
	%				&\eqsim \sum_{k=1}^\infty  
	%				\norm{\LocK \left(I-\Delta_g\right)^\frac{1}{2}u}_{L^2(J^{h_k},L^2)}^2\nonumber\\
	%				&\lesssim\, \norm{\left(I-\Delta_g\right)^\frac{1}{2}u}_{L^2(J^{h_k},L^2)}
	%				\eqsim \norm{u}_{L^2(J^{h_k},H^1)}\nonumber\\
	%				&\le \vert J^{h_k}\vert^\frac{1}{2}\norm{u}_{L^\infty(J^{h_k},H^1)}
	%				\lesssim\, \vert J^{h_k}\vert^\frac{1}{2}
	\end{align}
	Finally, the calculation 	
	\begin{align*}
	\lim_{p\to \infty}\frac{1}{p}\sum_{k=1}^\infty  
	2^{-\frac{3k}{p}}=\lim_{p\to \infty}\frac{1}{p}\left(\frac{1}{1-2^{-\frac{3}{p}}}-1\right)=\lim_{p\to \infty}\frac{1}{p\left(2^{\frac{3}{p}}-1\right)}=\frac{1}{3\log(2)}
	\end{align*}
	yields the boundedness of the function defined by $[6,\infty)\ni p\mapsto \frac{1}{p}\sum_{k=1}^\infty  
	2^{-\frac{3k}{p}}$
	and hence, 
	\begin{align}\label{linearGrowthGeometricSum}
	\sum_{k=1}^\infty  
	2^{-\frac{3k}{p}}\lesssim\, p.
	\end{align}
	Using the estimates  $\eqref{inhomLittlewoodPaleyEstimate}$ $\eqref{MiddleFrequenciesLittlewoodPaleyEstimate},$ $\eqref{intervalLengthEstimate},$ and $\eqref{linearGrowthGeometricSum}$ in $\eqref{finalEstimateHighNormsLemma}$, we get 	
	\begin{align*}
	\norm{u}_{L^{2}(J,L^p)}\lesssim\, 1+\left(\vert J\vert p\right)^\frac{1}{2},\qquad p\in[6,\infty),
	\end{align*}
	which implies the assertion. The proof of Proposition \ref{controlHighNorms} is thus completed.	
\end{proof}

We would like to continue with some remarks on seemingly natural extensions of the previous result to higher dimensions, nonlinear noise and  non-compact manifolds.

\begin{Remark}
	We would like to comment on the case of higher  dimensions $d\ge 4.$ The Strichartz-endpoint is $(2,\frac{2d}{d-2})$ and the use of Lemma \ref{LemmaEstimateNonlinearity} leads to the restriction $\alpha\le 1+\frac{2}{d-2}.$ The corresponding estimate in $\eqref{finalEstimateHighNormsLemma}$ has to be replaced by
	\begin{align*}%\label{finalEstimateHighNormsLemmaDimension}
	\norm{u}_{L^{2}(J,L^p)}
	\lesssim& \norm{\psi(\Delta_g)u}_{L^{2}(J,L^p)}+\sum_{k=1}^{k_0-1} \norm{\LocK u}_{L^{2}(J,L^p)}+\sum_{k=k_0}^\infty  
	2^{-\frac{k}{2}\left(\frac{d}{p}-\nu(d)\right)} \norm{\LocK u}_{L^2(J,H^1)}\nonumber\\
	&+\sum_{k=k_0}^\infty  \left[2^{-\frac{k}{2}\left(\frac{d}{p}-\nu(d)+\frac{1}{2}\right)} 			
	+2^{-\frac{k}{2}\left(\frac{d}{p}-\nu(d)+1\right)} 
	+2^{-\frac{k}{2} \left(\frac{d}{p}-\nu(d)+\frac{1}{2}\right)} \right]
	\end{align*} 
%	for $p\ge 2^*:=\frac{2 d}{d-2},$ where we set $\nu(d):=\frac{d-3}{2}.$ 
%	Hence, the convergence of the sums would require an upper bound for $p,$ which destroys the uniqueness proof below. In fact, this problem comes from the deterministic part of the equation $\eqref{ProblemStratonovich}$ and therefore also occurs in purely deterministic generalizations of Theorem 3 in \cite{Burq}.
	for $p\ge \frac{2 d}{d-2},$ where we set $\nu(d):=\frac{d-3}{2}.$ 
	Hence, the convergence of the sums requires an upper bound on p, which destroys the uniqueness proof below such that the case $d\ge 4$ remains an open problem. In fact, this problem occurs since the scaling condition for Strichartz exponents, Sobolev embeddings and Bernstein inequalities are more restrictive in higher dimensions and therefore, the restriction to $d=3$ is of deterministic nature.
\end{Remark}

\begin{Remark}
	In the proof of Proposition \ref{controlHighNorms}, we did not need the optimal estimates for the correction term $\mu$ and the stochastic integral. In fact, it is possible to generalize the argument and show the estimate 
	\begin{align*}
	\norm{u}_{L^2(J,L^p)}\lesssim\, 1+\left(\vert J\vert p\right)^\frac{1}{2}\qquad \text{a.s.},\quad p\in [6,\infty),
	\end{align*}
	 for martingale solutions of the equation 
	\begin{equation}\label{ProblemUniquenessNonlinear}
	%\begin{gathered}
	%(\operatorname{SNLS})\end{gathered}
	\left\{
	\begin{aligned}
	\df u(t)&= \left(\im \Delta_g u(t)-\im \lambda \vert u(t)\vert^{\alpha-1} u(t)+ \mu \left(\vert u(t)\vert^{2(\gamma-1)}u(t)\right) \right) \df t-\im B\left( \vert u(t)\vert^{\gamma-1}u(t)\right) \df W(t),\\
	u(0)&=u_0,
	\end{aligned}\right.
	\end{equation}
	with nonlinear noise of power $\gamma\in [1,2).$ However, we do not know if this equation has a solution, since the existence theory developed in \cite{ExistencePaper} only applies for $\gamma=1.$ Moreover, it is unclear how to apply these estimates in order to prove pathwise uniqueness since the arguments below rely on the linearity of the noise. Hence, the case of equation \eqref{ProblemUniquenessNonlinear} remains another open problem.
\end{Remark}

\begin{Remark}
	Let us comment on the case of possibly non-compact manifolds with bounded geometry. In the two dimensional setting, the Strichartz estimates from \cite{Bernicot} with an additional loss of $\varepsilon$ regularity were sufficient to prove uniqueness, see \cite{ExistencePaper}, Section 7. In fact, these estimates correspond to localized Strichartz estimates of the form
	\begin{align}\label{StrichartzHeat}
	\norm{t\mapsto e^{\im  t \Delta_g}\LocTwo{m}{\frac{1}{2}} x}_{L^q(J,L^p)}\le C_\varepsilon \norm{ x}_{L^2},\qquad \vert J\vert\le \beta_\varepsilon h^{1+\varepsilon},
	\end{align} 
	for all $\varepsilon>0$ and some $C_\varepsilon>0$ and $\beta_\varepsilon>0,$ where we denote $\psi_{m,a}(\lambda):= \lambda^m e^{-a \lambda}$	 for $m\in \N$ and $a>0.$ 
	%	we set 
	%	\begin{align*}
	%	\psi_{m,a}(\lambda):= \lambda^m e^{-a \lambda},\qquad \varphi_{m,a}(\lambda):=\int_{\lambda}^\infty \psi_{m,a}(t)\frac{\df t}{t},\qquad \lambda\ge 0.
	%	\end{align*} 
	A continuous version of the Littlewood-Paley inequality which can substitute \eqref{LittlewoodPaleyEstLp} is given by 
	\begin{align}\label{BernicotLittlewoodPaley}
	\norm{f}_{L^p}\eqsim \norm{\varphi_{m,a}(-\Delta_g)f}_{L^p}+\bigNorm{\left(\int_0^1 \vert\LocTwo{m}{a} f\vert^2 \frac{\df h}{h}\right)^\frac{1}{2}}_{L^p},\qquad f\in L^p(M),
	\end{align}
	% 	  	\begin{align}\label{BesovEmbedding}
	% 	  	\left(\int_0^1 \norm{\LocTwo{m}{a}f}_{L^p}^2 \frac{\df h}{h}\right)^\frac{1}{2}\lesssim \int_0^1 \norm{\LocTwo{m}{a}f}_{L^p} \frac{\df h}{h},\qquad f\in L^p(M),
	% 	  	\end{align}
	for $\varphi_{m,a}(\lambda):=\int_{\lambda}^\infty \psi_{m,a}(t)\frac{\df t}{t},$ see \cite{Bernicot}, Theorem 2.8.
	Based on \eqref{StrichartzHeat} and \eqref{BernicotLittlewoodPaley}, we can argue  similarly  as in the proof of Proposition \ref{controlHighNorms}
	and end up with the estimate
	\begin{align*}
	\norm{u}_{L^2(J,L^p)}\lesssim\, 1+\vert J\vert^\frac{1}{2}\left(\frac{p}{6-\varepsilon p}\right)^\frac{1}{2}\qquad \text{a.s.}
	\end{align*}
	for each $\varepsilon>0$  and $p\in [6,6 \varepsilon^{-1})$ with an implicit constant which goes to infinity for $\varepsilon\to 0.$ The upper bound on $p$ is due to the fact that the additional $\varepsilon$ in \eqref{StrichartzHeat} weakens the estimates of the critical term containing the derivative $\chi_j'$ of the temporal cut-off and enlarges the number of summands in \eqref{temporalSum}. As in the case of higher dimensions than $d=3,$ the uniqueness argument breaks down since a limit process $p\to \infty$ is no longer possible.
\end{Remark}

So far, we only used the topological properties of the noise, i.e.
\begin{align*}
	B\in \mathcal{L}\left(\Heins,\HS(Y,\Heins)\right),\qquad \mu\in\mathcal{L}(\Heins).
\end{align*}
 Now, the Stratonovich structure and the symmetry of the operators $B_m$ for $m\in\N$ come into play to prove 
the following representation formula for the $L^2$-distance of two solutions. 
\begin{Lemma}\label{SolutionDifferenceLemma}
	Let $d=3$ and $\alpha\in (1,3].$ Let $\left(\Omega,\F,\Prob,W,\Filtration,u_j\right),$ $j=1,2,$ be solutions of $\eqref{ProblemStratonovich}.$ 
%	with $F=F_\alpha^\pm$ for $\alpha>1$ and assume 
%	\begin{align*}
%	u_j\in L^{r}(\tilde{\Omega},L^\infty(0,T;H^1(M))).
%	\end{align*} 
%	for an $r>\max\{2\alpha, (\alpha-1)\alpha\}.$ 
	Then, we have 
	\begin{align}\label{SolutionDifferenceFormula}
	\norm{u_1(t)-u_2(t)}_{L^2}^2=&2 \int_0^t \Real \skpLzwei{u_1(s)-u_2(s)}{-\im \lambda \vert u_1(s)\vert^{\alpha-1} u_1(s)+\im \lambda \vert u_2(s)\vert^{\alpha-1} u_2(s)} \df s
	\end{align}
	almost surely for all $t\in[0,T].$
\end{Lemma}
Note that the RHS of \eqref{SolutionDifferenceFormula} only contains the terms induced by the nonlinearity. In particular, the stochastic integral vanishes, which will enable us to use the pathwise estimate from Proposition \ref{controlHighNorms} to prove uniqueness.

\begin{proof}
	We restrict ourselves to a formal argumentation. Similarly to \cite{ExistencePaper}, Proposition 6.5, our reasoning can be rigorously justified by a regularization procedure based on Yosida approximations $\Yosida:=\lambda\left(\lambda-\Delta_g\right)^{-1}$ for $\lambda>0.$ 
	The function $\mass: {L^2(M)} \to \R$ defined by $\mass(v):=\norm{v}_{L^2}^2$ is twice continuously Fr\'{e}chet-differentiable with 
	\begin{align*}
	\mass'[v]h_1&= 2 \Real \skpLzweiMf{v}{ h_1}, \qquad
	\mass''[v] \left[h_1,h_2\right]= 2 \Real \skpLzweiMf{ h_1}{h_2}
	\end{align*}
	for $v, h_1, h_2\in {L^2(M)}.$ We set $w:=u_1-u_2.$
	Then, a formal application of the It\^o formula yields 
		\begin{align}\label{LzweiDifferenceYosida}
		\norm{ w(t)}_{L^2}^2=&2 \int_0^t \Real \skpLzweiMf{w(s)}{\im \Delta_g w(s)-\im\vert u_1(s)\vert^{\alpha-1} u_1(s)+\im\vert u_2(s)\vert^{\alpha-1} u_2(s)} \df s\nonumber\\
		&+2 \int_0^t \Real \skpLzweiMf{w(s)}{\mu(w(s))} \df s- 2 \int_0^t \Real \skpLzweiMf{w(s)}{\im B w(s) \df W(s)}\nonumber\\
		&
		+\sumM \int_0^t   
		\Vert B_m w(s)\Vert_{L^2}^2\df s
		\end{align}
		almost surely for all $t\in[0,T].$ Since $\Delta_g$ is selfadjoint, we get $\Real \skpLzweiMf{w}{\im \Delta_g w}=0.$
		From the symmetry of $B_m,$ $m\in\N,$ we infer $\Real \skpLzweiMf{w}{\im B_m w}=0$ and thus, we obtain
		\begin{align*}
			\int_0^t \Real \skpLzweiMf{w(s)}{\im B w(s) \df W(s)}=0.
		\end{align*}
		 Moreover, we simplify 
		\begin{align*}
		2 \Real \skpLzweiMf{w(s)}{\mu(w(s))}=-\sumM \Real \skpLzweiMf{w(s)}{B_m^2 w(s)}
		=-\sumM \norm{B_m w(s)}_{L^2}^2. 
		\end{align*}
		Therefore, we have
		\begin{align*}
		\norm{ w(t)}_{L^2}^2=&2 \int_0^t \Real \skpLzweiMf{ w(s)}{-\im \vert u_1(s)\vert^{\alpha-1} u_1(s)+\im \vert u_2(s)\vert^{\alpha-1} u_2(s)} \df s
		\end{align*}
		almost surely for all $t\in[0,T].$

\end{proof}

%\begin{Remark}
%	With the same proof, one can show the mass conservation 
%	\begin{align*}
%	\norm{u(t)}_{L^2}=\norm{u_0}_{L^2} \qquad \text{a.s},\quad t\in [0,T],
%	\end{align*}
%	for martingale solutions of equation $\eqref{ProblemStratonovich}.$
%\end{Remark}

We close with the proof of our main Theorem \ref{Uniqueness3d}. We prove the uniqueness by applying a strategy developed by Yudovich, \cite{yudovich1963non}, for the Euler equation. In the context of the NLS, it was first used by Vladimirov in \cite{Vladimirov}, Ogawa and Ozawa in \cite{Ogawa} and \cite{OgawaOzawa}. They looked at $2D$ domains and used Trudinger type inequalities to control the growth of $L^p$-norms for $p\to \infty.$ A generalization of this argument to the stochastic case in $2D$ is straightforward and can be found in \cite{FhornungDiss}, Subsection 5.2.
Following Burq, G\'erard and Tzvetkov in the case without boundary, the Yudovich-strategy in combination with Strichartz estimates as an improvement of Trudinger's inequality was also applied it to the deterministic NLS on compact $3D$ manifolds with boundary by  Blair, Smith and Sogge in \cite{blairStrichartz}.

\begin{proof}[Proof of Theorem $\ref{Uniqueness3d}$]
	\emph{Step 1.}	
	Let us take two solutions $u_1,u_2\in L^2(\Omega,L^\infty(0,T;H^1(M))).$  Using Proposition \ref{controlHighNorms}, we choose a null set $N_1\in \F$ with
	\begin{align}\label{highNorms}
	\norm{u_j(\cdot,\omega)}_{L^2(J,L^{p})}\lesssim_\omega \,1+\left(\vert J\vert p\right)^\frac{1}{2},\qquad \omega\in \Omega\setminus N_1, 
	\end{align}
	for each interval $J\subset [0,T]$ and $p\ge 6.$
	By Corollary \ref{SolutionDifferenceLemma}, we choose a null set $N_2\in F$ such that
	\begin{align}\label{solutionDifference}
	\norm{u_1(t)-u_2(t)}_{L^2}^2=&2 \int_0^t \Real \skpLzwei{u_1(s)-u_2(s)}{-\im \lambda \vert u_1(s)\vert^{\alpha-1} u_1(s)+\im \lambda \vert u_2(s)\vert^{\alpha-1} u_2(s)} \df s
	\end{align}
	holds on $\Omega\setminus N_2$ for all $t\in [0,T].$
	In particular, this leads to the weak differentiability of the map
	$G:=\norm{u_1-u_2}_{L^2}^2$ on $\Omega\setminus N_2$ and to the estimate
		\begin{align}\label{DifferenceFormulaModelUniqueness2}
			\vert G'(t)\vert=&\left\vert 2 \Real \skpLzwei{u_1(s)-u_2(s)}{-\im\lambda  \vert u_1(s)\vert^{\alpha-1} u_1(s)+\im\lambda  \vert u_2(s)\vert^{\alpha-1} u_2(s)}\right\vert\nonumber\\
%			\le& \int_0^t \int_M \vert u_1(s,x)-u_2(s,x)\vert \left(\big\vert \vert u_1(s,x)\vert^{\alpha-1} u_1(s,x)- \vert u_2(s,x)\vert^{\alpha-1} u_2(s,x)\big\vert\right) \df x\df s\nonumber\\
			\lesssim& \int_M \vert u_1(s,x)-u_2(s,x)\vert^{2} \left(\vert u_1(s,x)\vert^{\alpha-1}+\vert u_2(s,x)\vert^{\alpha-1}\right) \df x.
%			\le& \int_0^t \norm{u_1(s)-u_2(s)}_{L^{2p'}}^2 \norm{y(s)}_{L^{(\alpha-1)p}}^{\alpha-1} \df s,
%			\qquad t\in[0,T].
		\end{align}
		The Sobolev embedding $H^1(M)\hookrightarrow L^6(M)$ yields $u_j\in L^\infty(0,T;L^6(M))$, $j=1,2$, almost surely. Moreover, we have the mild representation
		\begin{align*}
		\im u_j(t)=&\im e^{\im t\Delta_g}u_0+\int_0^t e^{\im (t-\tau)\Delta_g}\lambda \vert u_j(\tau)\vert^{\alpha-1} u_j(\tau)\df \tau
		+\im \int_0^t e^{\im (t-\tau)\Delta_g}\mu(u_j(\tau))\df \tau\\
		&+\int_0^t e^{\im (t-\tau)\Delta_g}B(u_j(\tau))\df W(\tau)
		\end{align*}
		almost surely for all $t\in [0,T]$ in $H^{-1}(M)$ for $j=1,2.$ As a consequence of $\alpha \in (1,3]$ and $u_j\in L^\infty(0,T;L^6(M)),$ each of the terms on the RHS is in $L^2(M).$ In particular,  we obtain $u_j\in C([0,T],\Lzwei)$, $j=1,2$, almost surely and thus, we can take another null set 	$N_3\in F$ such that
		\begin{align*}
		u_j\in L^\infty(0,T;L^6(M))\cap C([0,T],\Lzwei)\quad \text{on} \quad \Omega\setminus N_3.
		\end{align*}
		Now, we define 
		$\Omega_1:=\Omega\setminus \left(N_1\cup N_2\cup N_3\right)$	
		and fix $\omega\in \Omega_1.$	We take a sequence $\left(p_n\right)_{n\in\N}\in [6,\infty)^\N$ with $p_n\to \infty$ as $n\to \infty.$ 
	We fix $n\in\N$ and define $q_n:=\frac{p_n}{\alpha-1}.$ 
	By the estimate \eqref{DifferenceFormulaModelUniqueness2} and H\"older's inequality with exponents $\frac{1}{q_n'}+\frac{1}{q_n}=1$, we get 
	\begin{align*}
	\vert G'(t)\vert
	\lesssim& \norm{u_1(t)-u_2(t)}_{L^{2{q'_n}}}^2 \bigNorm{\vert u_1(t)\vert^{\alpha-1}+\vert u_2(t)\vert^{\alpha-1}}_{L^{{q_n}}} ,
	\qquad t\in[0,T].
	\end{align*}
	The choice of ${q_n}$ yields $2{q'_n}\in [2,6]$ and for $\theta:=\frac{3}{2{q_n}}\in (0,1),$ we have $\frac{1}{2 {q'_n}}=\frac{1-\theta}{2}+ \frac{\theta}{6}.$ Hence, we obtain
	\begin{align*}
	\norm{u_1-u_2}_{L^{2{q'_n}}}^2
	\le \norm{u_1-u_2}_{L^2}^{2-\frac{3}{{q_n}}} \norm{u_1-u_2}_{L^6}^\frac{3}{{q_n}}
	\le \norm{u_1-u_2}_{L^2}^{2-\frac{3}{{q_n}}} \norm{u_1-u_2}_{L^\infty(0,T;L^6)}^\frac{3}{{q_n}}
	\end{align*} 
	by interpolation. We choose a constant $C_1>0$ such that 
	\begin{align*}
	\norm{u_1}_{L^\infty(0,T;L^6)}+\norm{u_2}_{L^\infty(0,T;L^6)}\le C_1,
	\end{align*}
	which leads to 
	\begin{align}\label{EstimateG}
	\vert G'(t)\vert
	\lesssim&\, C_1^\frac{3}{{q_n}} G(t)^{1-\frac{3}{{2q_n}}}  \left[\norm{u_1(t)}_{L^{{p_n}}}^{\alpha-1}+\norm{u_2(t)}_{L^{{p_n}}}^{\alpha-1}\right].
	%	\nonumber\\
	%	\le& \norm{u_1-u_2}_{L^\infty(0,T;L^6)}^\frac{3}{{p_n}} \int_0^t G(s)^{1-\frac{3}{2{p_n}}} 
	%	\left[\norm{u_1(s)}_{L^{(\alpha-1){p_n}}}^{\alpha-1}+\norm{u_2(s)}_{L^{(\alpha-1){p_n}}}^{\alpha-1}\right] \df s	
	\end{align}

	\emph{Step 2.}
	We argue by contradiction and assume that there is $t_2\in [0,T]$ with $G(t_2)>0.$
	By the continuity of $G$, we get
	\begin{align}\label{Assumption}
	\exists t_1\in [0,t_2): G(t_1)=0 \quad\text{and}\quad \forall t\in (t_1,t_2): G(t)>0.
	\end{align}
	We set $J_\varepsilon:= (t_1,t_1+\varepsilon)$ with $\varepsilon\in (0,t_2-t_1)$ to be chosen later. By the weak chain rule (see \cite{gilbargTrudinger}, Theorem 7.8) and $\eqref{EstimateG},$ we get 
	\begin{align*}
	G(t)^\frac{3}{2 {q_n}}=\frac{3}{2 {q_n}} \int_{t_1}^t G'(s)  G(s)^{\frac{3}{2 {q_n}}-1} \df s
	&\lesssim \frac{3}{2 {q_n}} C_1^{\frac{3}{{q_n}}}\int_{t_1}^t \left[\norm{u_1(s)}_{L^{{p_n}}}^{\alpha-1}+\norm{u_2(s)}_{L^{{p_n}}}^{\alpha-1}\right] \df s,\qquad t\in J_\varepsilon.
	%		\\
	%		&\le\frac{3}{2 p} C_1^{\frac{3}{p}} \left(\norm{u_1}_{L^{\alpha-1}(J_\varepsilon,L^{(\alpha-1)p})}^{\alpha-1}+\norm{u_1}_{L^{\alpha-1}(J_\varepsilon,L^{(\alpha-1)p})}^{\alpha-1}\right)
	\end{align*}
%	almost surely for $t\in J_\varepsilon.$ 
	By another application of the H\"older inequality with exponents $\frac{2}{\alpha-1}$ and $\frac{2}{3-\alpha},$ we infer that
	\begin{align*}
	G(t)^\frac{3}{2 {q_n}}
	&\lesssim \frac{3}{2 {q_n}} C_1^{\frac{3}{{q_n}}}
	\Big[\norm{u_1}_{L^2(t_1,t;L^{p_n})}^{\alpha-1}+\norm{u_2}_{L^2(t_1,t;L^{p_n})}^{\alpha-1}\Big]\varepsilon^{\frac{3-\alpha}{2}},\qquad t\in J_\varepsilon.
	%		\\
	%		&\le\frac{3}{2 p} C_1^{\frac{3}{p}} \left(\norm{u_1}_{L^{\alpha-1}(J_\varepsilon,L^{(\alpha-1)p})}^{\alpha-1}+\norm{u_1}_{L^{\alpha-1}(J_\varepsilon,L^{(\alpha-1)p})}^{\alpha-1}\right)
	\end{align*}	
	Now, we are in the position to apply \eqref{highNorms} and we obtain
	\begin{align*}
	G(t)^\frac{3}{2 {q_n}}
	&\lesssim \frac{3}{2 {q_n}} C_1^{\frac{3}{{q_n}}}\left(1+(\varepsilon p_n)^\frac{\alpha-1}{2}\right)\varepsilon^{\frac{3-\alpha}{2}},\qquad t\in J_\varepsilon.
	%	\Big[\norm{u_1}_{L^2(t_1,t;L^{p_n})}^{\alpha-1}+\norm{u_2}_{L^2(t_1,t;L^{p_n})}^{\alpha-1}\Big]\varepsilon^{\frac{3-\alpha}{2}}.
	%		\\
	%		&\le\frac{3}{2 p} C_1^{\frac{3}{p}} \left(\norm{u_1}_{L^{\alpha-1}(J_\varepsilon,L^{(\alpha-1)p})}^{\alpha-1}+\norm{u_1}_{L^{\alpha-1}(J_\varepsilon,L^{(\alpha-1)p})}^{\alpha-1}\right)
	\end{align*}	
	In particular, there is a constant $C>0$ such that for all $t\in J_\varepsilon$ it holds that
	\begin{align}\label{UniquenessBeforeEstimateHighNorms}
	G(t)
	&\le C_1^2  \left(\frac{3 C}{2 {q_n}}\left(1+(\varepsilon (\alpha-1)q_n)^\frac{\alpha-1}{2}\right) \varepsilon^{\frac{3-\alpha}{2}}\right)^\frac{2 {q_n}}{3}\nonumber\\
	&\le C_1^2  \left(\frac{3C }{2 {q_n}}\left(1+\varepsilon^\frac{\alpha-1}{2} (\alpha-1)q_n\right) \varepsilon^{\frac{3-\alpha}{2}}\right)^\frac{2 {q_n}}{3} =:b_n,
	\end{align}
	where we used $p_n:={q_n}(\alpha-1)$ and $\frac{\alpha-1}{2}\in (0,1].$\\

	\emph{Step 3.} 	We aim to show that the sequence $\left(b_n\right)_{n\in \N}$ on the RHS of \eqref{UniquenessBeforeEstimateHighNorms} converges to $0$ for $\varepsilon$ sufficiently small. Then, we have proved $G(t)=0$ for all $t\in J_\varepsilon$ which contradicts $\eqref{Assumption}.$ Hence, we have  $u_1(t)=u_2(t)$ almost surely for all $t\in [0,T].$
	
	To this end, we choose $\varepsilon \in (0,\min\{t_2-t_1,\frac{2}{3C (\alpha-1)}\}) .$ Then, 
	\begin{align*}
	b_n=&C_1^2  \left(\frac{3C }{2 {q_n}}\left(1+\varepsilon^\frac{\alpha-1}{2} (\alpha-1)q_n\right) \varepsilon^{\frac{3-\alpha}{2}}\right)^\frac{2 {q_n}}{3}\\
	=&C_1^2  \left(\frac{3C\varepsilon(\alpha-1) }{2}\right)^\frac{2 {q_n}}{3}\left(\frac{1}{\varepsilon^\frac{\alpha-1}{2} (\alpha-1)q_n}+1\right)^\frac{2 {q_n}}{3}
	\xrightarrow{n\to \infty}0.
	\end{align*}
	The proof of Theorem \ref{Uniqueness3d} is thus completed.
\end{proof}

\section*{Acknowledgement}
We gratefully acknowledge financial support by the
Deutsche
Forschungsgemeinschaft (DFG) through CRC 1173.

%\input{ContinuousDependance}
%\bibliographystyle{alpha}
%\bibliographystyle{abbrv}
%\bibliography{ReferencesAll}

\Addresses
\end{document}